\documentclass[11pt]{amsart}

\usepackage{evearticle}
\DeclareMathOperator\Mod{Mod}

\begin{document}

\title{Quasisymmetric rigidity of the Brownian sphere}
\author{Jason Miller and Yi Tian}

\date{\today}

\begin{abstract}
    The Brownian sphere, also known as the Brownian map, is a canonical random metric measure space homeomorphic to the two-dimensional sphere $\BS^2$. It can be interpreted as the uniform measure on surfaces homeomorphic to $\BS^2$, in the sense that it arises as the scaling limit of many natural models of random planar maps chosen uniformly from a given class. It is also equivalent to the $\sqrt{8/3}$-Liouville quantum gravity sphere. We prove that the Brownian sphere is quasisymmetrically rigid, meaning that, almost surely, it has no nontrivial quasisymmetric automorphisms. We also show that two independent Brownian spheres are almost surely not quasisymmetrically equivalent. Our argument also gives a new proof that the conformal structure of the Brownian sphere is almost surely determined by its metric structure.
\end{abstract}

\maketitle

\tableofcontents

\setlength{\parindent}{0pt}
\setlength{\parskip}{0.5\baselineskip plus 1pt minus 1pt}

\section{Introduction}

In the present paper, we study the quasisymmetric geometry of the Brownian sphere, a random metric measure space which is almost surely homeomorphic to the two-dimensional sphere $\BS^2$ but has Hausdorff dimension $4$. More precisely, we prove rigidity results for the Brownian sphere under quasisymmetric mappings. We first recall some background and fix terminology.

The \emph{Brownian sphere}, also called the \emph{Brownian map}, can be viewed heuristically as a sample from the ``uniform measure'' on the space of all Riemannian metrics on $\BS^2$. Such a measure does not have an obvious meaning, since the space of all Riemannian metrics on $\BS^2$ is infinite-dimensional. The Brownian sphere can also be described as the $\gamma$-\emph{Liouville quantum gravity (LQG)} sphere with $\gamma = \sqrt{8/3}$ (or, equivalently, with $\Bc = 0$). LQG surfaces were first introduced non-rigorously in the physics literature by Polyakov \cite{Polyakov:1981rd} in the context of bosonic string theory, as a string-theoretic analogue of the Feynman path integral.

Recall that a \emph{planar map} is a proper embedding of a connected planar graph into $\BS^2$, considered up to orientation-preserving homeomorphisms. Random planar maps may be viewed as discrete models of two-dimensional quantum gravity. It was shown independently by Le Gall \cite{MR3112934} and Miermont \cite{MR3070569} that the Brownian sphere is the Gromov--Hausdorff--Prokhorov scaling limit of quadrangulations, i.e., planar maps whose faces all have degree four, chosen uniformly at random as the number of faces tends to $\infty$. Subsequent works showed that many other natural families of random planar maps also converge to the Brownian sphere; see \cite{MR3256874,MR3498001,MR3706731,MR4315765}. The Brownian sphere is also the scaling limit of random hyperbolic surfaces with genus zero and many cusps \cite{2025arXiv250818792B}. The equivalence between the Brownian sphere and the $\sqrt{8/3}$-LQG sphere was established in a series of works \cite{MR3572845,MR4050102,MR4225028,MR4348679,MR4242633}.

\emph{Quasiconformal mappings} are generalizations of conformal mappings. There are several equivalent definitions of quasiconformal mappings, including analytic, metric, and geometric definitions; see \Cref{subsection:background-quasiconformal}. Quasiconformal mappings have important applications in many areas, such as Teichm\"uller theory \cite{MR2241787,MR1730906,MR1215481}, complex dynamics \cite{MR819553,MR806415,MR1230383,MR3445628}, and geometric group theory and analysis on metric spaces \cite{MR919829,MR2116315,MR1654771,MR1800917,MR1919393}. After imposing a suitable normalization, the measurable Riemann mapping theorem identifies quasiconformal homeomorphisms with Beltrami coefficients $\mu \in L^\infty$ with $\lVert\mu\rVert_\infty < 1$.

\emph{Quasisymmetric mappings} are a metric-space variant of quasiconformal mappings, and are sometimes better suited to general metric spaces \cite{MR595180}. Suppose that $\eta \colon \BR_{\ge 0} \to \BR_{\ge 0}$ is a homeomorphism.  Then a homeomorphism $f \colon (X, D_X) \to (Y, D_Y)$ between two metric spaces is \emph{$\eta$-quasisymmetric} if it satisfies
\begin{equation*}
    \frac{D_Y(f(x), f(y))}{D_Y(f(x), f(z))} \le \eta\!\left(\frac{D_X(x, y)}{D_X(x, z)}\right)
\end{equation*}
for all distinct $x, y, z \in X$.   In this case, $\eta$ is called the \emph{distortion function}.  The class of quasisymmetric mappings is closed under both inversion and composition. Thus the existence of a quasisymmetric homeomorphism defines an equivalence relation on metric spaces, and we say that two metric spaces are \emph{quasisymmetrically equivalent} if such a homeomorphism exists between them. An important theme in hyperbolic geometry and geometric group theory is that quasi-isometries between hyperbolic manifolds or groups induce quasisymmetric mappings between their boundaries at infinity. This principle underlies rigidity results such as Mostow's rigidity theorem \cite{MR236383} and the quasiconformal rigidity of Kleinian groups \cite{MR806415,MR1726737}.

The classical uniformization theorem states that every simply connected two-dimensional Riemannian manifold is conformally equivalent to the upper half-plane, the complex plane, or the Riemann sphere. The quasisymmetric uniformization problem asks for an analogous classification of more general metric spaces up to quasisymmetric equivalence. The theorem of Bonk--Kleiner \cite{MR1930885} gives sufficient conditions for a metric space to be quasisymmetrically equivalent to $\BS^2$; see also \cite{MR3608292,MR4073230,MR4956983,MR4608329}. An analogous result for round carpets was proved in \cite{MR2854086}.

It was shown in \cite{MR4242630} that the Brownian sphere is almost surely not quasisymmetrically equivalent to $\BS^2$. An important invariant for classifying metric spaces up to quasisymmetric equivalence is the \emph{conformal dimension} \cite{MR1024425}. By definition, the conformal dimension of a metric space is the infimum of the Hausdorff dimensions of all metric spaces quasisymmetrically equivalent to it. It was proved in \cite{2026arXiv260324473M} that the conformal dimension of the Brownian sphere is almost surely equal to two. The next natural question is whether all Brownian sphere instances belong to the same quasisymmetric equivalence class, or if the class itself varies with the realization (i.e., is random). Our main results answer this question.

\begin{theorem}\label{thm:main-ind}
    Almost surely, there is no quasisymmetric homeomorphism between two independent Brownian spheres.
\end{theorem}

\begin{theorem}\label{thm:main-aut}
    Almost surely, a Brownian sphere has no nontrivial quasisymmetric automorphisms.
\end{theorem}

We note that \Cref{thm:main-aut} implies that there can be at most one quasisymmetric homeomorphism between any given pair of Brownian sphere instances. Indeed, the composition of two such homeomorphisms would give a quasisymmetric automorphism of one of the Brownian spheres. If the two Brownian spheres are independent, then \Cref{thm:main-ind} implies that there is no quasisymmetric homeomorphism between them. However, there are couplings of pairs of Brownian sphere instances which are not isometric and for which there exists a nontrivial quasisymmetric homeomorphism between them. For example, if we realize the Brownian sphere using the GFF-like distribution associated with a $\sqrt{8/3}$-LQG sphere, then one can resample a single coefficient in the series expansion of the field to obtain another $\sqrt{8/3}$-LQG sphere. In this coupling, the identity map on the underlying parameterizing domain is quasisymmetric.

In the context of the quasisymmetric uniformization problem, \Cref{thm:main-ind} implies that the Brownian sphere measure on the space of metric spaces homeomorphic to $\BS^2$ assigns zero mass to each quasisymmetric equivalence class; equivalently, the induced measure on the quotient by quasisymmetric equivalence is non-atomic. In particular, there does not exist a single deterministic model space to which every Brownian sphere instance is quasisymmetrically equivalent. This contrasts with the class of Ahlfors regular and linearly locally connected surfaces considered in \cite{MR1930885}, which are shown there to be quasisymmetrically equivalent to~$\BS^2$.

Quasisymmetric rigidity has also been studied for deterministic fractals, including round carpets \cite{MR2503988} and square Sierpi\'nski carpets \cite{MR3010807,MR4140760}. In \cite{MR3010807,MR4140760}, it was proved that:
\begin{itemize}
    \item If $p, q \ge 3$ are odd integers, then the Sierpi\'nski carpets $S_p$ and $S_q$ are quasisymmetrically equivalent if and only if $p = q$. 
    \item Every quasisymmetric automorphism $S_p \to S_p$ is an isometry.
\end{itemize}
In this sense, \Cref{thm:main-ind} is reminiscent of the first statement, while \Cref{thm:main-aut} is reminiscent of the second.

The quasisymmetric uniformization problem also has important connections with geometric group theory. For example, quasi-isometries between Gromov hyperbolic groups induce quasisymmetric homeomorphisms between their Gromov boundaries equipped with visual metrics \cite{MR1771428}. The Gromov boundary of a Gromov hyperbolic group is a crucial invariant for studying the properties of the group. Cannon's conjecture \cite{MR1301392} asserts that if the Gromov boundary of a Gromov hyperbolic group is homeomorphic to $\BS^2$, then it is quasisymmetrically equivalent to $\BS^2$. The Kapovich--Kleiner conjecture \cite{MR1834498} asserts that if the Gromov boundary of a Gromov hyperbolic group is homeomorphic to the standard Sierpi\'nski carpet, then it is quasisymmetrically equivalent to a round carpet.

Since a non-elementary (i.e., not virtually cyclic) Gromov hyperbolic group acts nontrivially on its Gromov boundary by quasisymmetric homeomorphisms with respect to any visual metric (cf., e.g., \cite[Proposition~1.18]{MR4273635}), \Cref{thm:main-aut} implies the following.

\begin{corollary}\label{cor:group}
    Almost surely, the Brownian sphere is not quasisymmetrically equivalent to the Gromov boundary of any Gromov hyperbolic group equipped with a visual metric.
\end{corollary}

\Cref{cor:group} is not new. It can also be derived from the fact that the Brownian sphere is almost surely not doubling, proved in \cite{MR4242630}. More broadly, it is natural to ask whether there is any group-theoretic construction which gives a boundary related, in a suitable sense, to the Brownian sphere. We note, however, that our main results give a way to rule this out in cases where the group acts on its boundary by quasisymmetric homeomorphisms. For example, this applies to some Bowditch boundaries of relatively hyperbolic groups equipped with suitable visual-type metrics. Such boundaries are not doubling in general, so this conclusion does not follow directly from \cite{MR4242630}.

An important ingredient in the proof of \Cref{thm:main-ind,thm:main-aut} is the fact that the Brownian sphere carries a canonical conformal structure. It was proved in \cite{MR4242633} that this conformal structure is almost surely determined by the metric. Moreover, by \cite{TuttePVBDtoLQG}, it can be recovered explicitly from the metric as the limit of the Tutte embeddings of the Poisson--Voronoi tessellations associated with the Brownian sphere. To the best of our knowledge, \cite{MR4242633,TuttePVBDtoLQG} provide the only previously known proofs of this result, with the former relying on conformal removability. By adapting the argument used to prove \Cref{thm:main-ind,thm:main-aut}, we obtain a new proof.

\begin{corollary}\label{cor:measurability}
    The metric structure of the Brownian sphere almost surely determines its conformal structure up to complex conjugation.
\end{corollary}

The proof of \Cref{thm:main-ind,thm:main-aut} has two main steps. In \Cref{section:quasiconformality}, we first show that, almost surely, every quasisymmetric homeomorphism between two coupled (not necessarily independent) Brownian spheres is quasiconformal with respect to their conformal structures. Then, in \Cref{section:main-proof}, we prove that, when the two Brownian spheres are independent, such a quasisymmetric homeomorphism cannot exist. This strategy is motivated by the study of quasisymmetric rigidity for deterministic fractals in Euclidean space, such as round carpets \cite{MR2503988} and square Sierpi\'nski carpets \cite{MR3010807}, where one often uses the fact that a quasisymmetric mapping extends to a quasiconformal mapping of the ambient space. Finally, in \Cref{section:measurability}, we explain how to adapt the arguments of \Cref{section:quasiconformality,section:main-proof} to prove \Cref{cor:measurability}. Indeed, the results of \Cref{section:quasiconformality} imply that the identity map between the Brownian sphere equipped with any two possible conformal structures is quasiconformal. Then, by modifying the argument of \Cref{section:main-proof}, we show that this identity map must in fact be conformal.

We emphasize that one can also define quasiconformality using the metrics of the Brownian spheres.  \emph{Throughout the remainder of the present paper, quasiconformality will always mean quasiconformality with respect to the conformal structures of the Brownian spheres.}

We will find it convenient in this paper to first work in the setting of the Brownian plane (an unbounded variant of the Brownian sphere which may be obtained as a tangent cone of the Brownian sphere \cite{TBP}) and the whole-plane GFF, and to prove rigidity in this setting. These objects are often more convenient to study because they have scale and translation invariance properties, as well as independence properties across scales; see \Cref{subsection:independence}. We then transfer the results to the Brownian sphere at the end of the proof using local absolute continuity.

Another important ingredient in the second step of the proof is the fact that a quasiconformal mapping is differentiable, with invertible Jacobian matrix, at Lebesgue almost every point. This is a deep fact which follows from the equivalence of the different definitions of quasiconformality; see \Cref{subsection:background-quasiconformal}. In particular, at a Lebesgue typical point, the mapping is locally close to a linear mapping. This will allow us to derive a contradiction.

We also note that the proof strategy is reminiscent of the proof of Mostow's rigidity theorem (see, e.g., \cite{MR1168043} for more details). There, one first shows that an isomorphism of the fundamental groups induces an equivariant quasi-isometry between the universal covers of the two hyperbolic manifolds, and also an equivariant quasiconformal mapping between the spheres at infinity. Then, using analytic results on quasiconformal mappings, one deduces that the quasi-isometry between the universal covers must in fact be an isometry.

\subsection*{Acknowledgements}

J.M.~received support from ERC consolidator grant ARPF (Horizon Europe UKRI G120614). Y.T.~was supported by a Cambridge International Scholarship from the Cambridge Trust. 

\section{Background}

\subsection{Notation and conventions}

The notation $\BC$ (resp.~$\BR$; $\BQ$; $\BZ$; $\BN$) will be used to denote the set of complex numbers (resp.~real numbers; rational numbers; integers; positive integers). The notation $\lVert\bullet\rVert$ will be used to denote the Euclidean norms of complex numbers. We shall write $\widehat\BC \defeq \BC \cup \{\infty\}$ and $\BD$ for the open unit disk. For $a < b$, we shall write $[a, b]_\BZ \defeq [a, b] \cap \BZ$.

Let $(X, D)$ be a metric space. We shall write $\diam(X; D) \defeq \sup\{D(x, y) : x, y \in X\}$. Let $x \in X$ and $0 < s < t$. We shall write $B_t(x; D) \defeq \{y \in X : D(x, y) < t\}$ and $A_{s,t}(x; D) \defeq \{y \in X : s < D(x, y) < t\}$. When $D$ is the Euclidean metric, we omit it from our notation.

We call an open set $U \subset \BC$ a \emph{simply connected domain} if $U$ is connected and $\BC \setminus U$ is connected and unbounded. We call a connected open set $A \subset \BC$ a \emph{doubly connected domain} if $\BC \setminus A$ consists of one unbounded connected component and one bounded connected component, and the bounded component is not a singleton. For a doubly connected domain $A \subset \BC$, we define the \emph{conformal modulus} $\Mod(A)$ of $A$ to be the unique number $M > 0$ such that $A$ is conformally equivalent to $A_{\re^{-2\pi M},1}(0)$. For open subsets $U, V \subset \BC$, we shall write $U \Subset V$ if $\overline U \subset V$.

Let $\{E_t\}_{t > 0}$ be a family of events. Then:
\begin{itemize}
    \item We shall say that $E_t$ occurs with \emph{polynomially high probability} as $t \to 0$ (resp.~$t \to \infty$) if there exist $\alpha, C > 0$ such that $\BP\lbrack E_t\rbrack \ge 1 - Ct^\alpha$ (resp.~$\BP\lbrack E_t\rbrack \ge 1 - Ct^{-\alpha}$) for all $t > 0$. 
    \item We shall say that $E_t$ occurs with \emph{superpolynomially high probability} as $t \to 0$ (resp.~$t \to \infty$) if for each $\alpha > 0$, there exists $C = C(\alpha) > 0$ such that $\BP\lbrack E_t\rbrack \ge 1 - Ct^\alpha$ (resp.~$\BP\lbrack E_t\rbrack \ge 1 - Ct^{-\alpha}$) for all $t > 0$. 
\end{itemize}

\subsection{The Brownian sphere and the Brownian plane}

In the present subsection, we recall the construction of the Brownian sphere and the Brownian plane, together with their breadth-first explorations.

Let $\{X_t\}_{t \in [0, 1]}$ be a normalized Brownian excursion. Conditional on $X$, let $\{Z_t\}_{t \in [0, 1]}$ be the centered Gaussian process whose covariance is given by
\begin{equation*}
    \BE\lbrack Z_s Z_t\rbrack = \inf_{u \in [s, t]} X_u, \quad \forall 0 \le s \le t \le 1. 
\end{equation*}
For $0 \le s \le t \le 1$, set
\begin{equation*}
    D^\circ(s, t) \defeq Z_s + Z_t - 2\left(\inf_{u \in [s, t]} Z_u \vee \inf_{u \in [0, s] \cup [t, 1]} Z_u\right), \quad \forall 0 \le s \le t \le 1. 
\end{equation*}
Let $D$ be the largest pseudo-metric on $[0,1]$ such that $D \le D^\circ$ and such that $D(s,t) = 0$ whenever $D^\circ(s,t) = 0$. We then define $\SCS \defeq [0,1]/\sim$, where $s \sim t$ if and only if $D(s,t) = 0$. The pseudo-metric $D$ induces a metric $D_\SCS$ on $\SCS$. We also let $\SM_\SCS$ be the pushforward of Lebesgue measure on $[0,1]$ under the quotient mapping $\mathop{\mathrm{proj}} \colon [0,1] \to \SCS$. The \emph{Brownian sphere with unit area} is the two-pointed metric measure space $(\SCS, D_\SCS, \SM_\SCS; \mathop{\mathrm{proj}}(0), \mathop{\mathrm{proj}}(t_\ast))$, where $t_\ast$ is the almost surely unique time at which $Z$ attains its minimum.

Almost surely, $(\SCS, D_\SCS)$ is a geodesic metric space homeomorphic to $\BS^2$ (cf.~\cite{MR2438999,MR2399286}), and its Hausdorff dimension is $4$ (cf.~\cite{MR2336042}). Conditional on $(\SCS, D_\SCS, \SM_\SCS)$, the two marked points $\mathop{\mathrm{proj}}(0)$ and $\mathop{\mathrm{proj}}(t_\ast)$ are independent samples from $\SM_\SCS$ (cf.~\cite{GeoBMap}). Moreover, $\SM_\SCS$ is almost surely equal to the Hausdorff measure associated with the gauge function $r \mapsto r^4\log(\log(1/r))$ (cf.~\cite{MR4474537}).

The \emph{(unconditioned) Brownian sphere measure} is obtained by randomizing the total area. More precisely, it is the pushforward of the measure $\BP \otimes c A^{-3/2} \, \rd A$, for a constant $c > 0$, under the assignment $(\SCS, D_\SCS, \SM_\SCS; x, y) \mapsto (\SCS, A^{1/4}D_\SCS, A\SM_\SCS; x, y)$, where $\BP$ denotes the law of the Brownian sphere with unit area. Equivalently, one may carry out the same construction as above with It\^o's excursion measure in place of the normalized Brownian excursion.

The Brownian plane is the infinite-volume analogue of the Brownian sphere. Let $\{X_t\}_{t \in \BR}$ be such that $\{X_t\}_{t \ge 0}$ and $\{X_{-t}\}_{t \ge 0}$ are independent three-dimensional Bessel processes started from zero. Conditional on $X$, let $\{Z_t\}_{t \in \BR}$ be the centered Gaussian process with covariance
\begin{equation*}
    \BE\lbrack Z_s Z_t\rbrack = 
    \begin{cases}
        \inf_{u \in [s, t]} X_u & \text{if } s \le t \le 0 \text{ or } 0 \le s \le t; \\
        \inf_{u \in (-\infty, s] \cup [t, \infty)} X_u & \text{if } s \le 0 \le t. 
    \end{cases} 
\end{equation*}
For $s \le t$, define
\begin{equation*}
    D^\circ(s, t) \defeq Z_s + Z_t - 2\inf_{u \in [s, t]} Z_u, \quad \forall s \le t. 
\end{equation*}
Let $D$ be the largest pseudo-metric on $\BR$ such that $D \le D^\circ$ and such that $D(s,t) = 0$ whenever $D^\circ(s,t) = 0$. We define $\SCC \defeq \BR/\sim$, where $s \sim t$ if and only if $D(s,t) = 0$. Let $D_\SCC$ be the induced metric on $\SCC$, and let $\SM_\SCC$ be the pushforward of Lebesgue measure on $\BR$ under the quotient mapping $\mathop{\mathrm{proj}} \colon \BR \to \SCC$. The \emph{Brownian plane} is the pointed metric measure space $(\SCC, D_\SCC, \SM_\SCC; \mathop{\mathrm{proj}}(0))$.

The Brownian plane appears as the tangent cone of the Brownian sphere: $(\SCS, \lambda D_\SCS, \lambda^4\SM_\SCS, \mathop{\mathrm{proj}}(0))$ converges in law as $\lambda \to \infty$ to $(\SCC, D_\SCC, \SM_\SCC; \mathop{\mathrm{proj}}(0))$ with respect to the local Gromov--Hausdorff--Prokhorov topology (cf.~\cite{TBP}). It also satisfies the following scaling invariance: For each deterministic $\lambda > 0$, the rescaled pointed metric measure space $(\SCC, \lambda D_\SCC, \lambda^4\SM_\SCC; \mathop{\mathrm{proj}}(0))$ has the same law as $(\SCC, D_\SCC, \SM_\SCC; \mathop{\mathrm{proj}}(0))$.

Let $(\SCC, D_\SCC, \SM_\SCC; x)$ be a Brownian plane. In contrast with Euclidean space, the complement of a metric ball $B_t(x; D_\SCC)$ has countably many connected components. We define the filled metric ball $B_t^\bullet(x; D_\SCC)$ to be the complement of the unbounded connected component of $\SCC \setminus B_t(x; D_\SCC)$. Equivalently, one obtains $B_t^\bullet(x; D_\SCC)$ by filling in all bounded complementary components of $B_t(x; D_\SCC)$.

By \cite{Hull-TBP}, there is a c\`adl\`ag process $\{Y_s\}_{s \le 0}$ with only positive jumps, almost surely determined by $(\SCC, D_\SCC, \SM_\SCC; x)$, such that for each deterministic $t \ge 0$, in probability,
\begin{equation*}
    Y_{-t} = \lim_{\varepsilon \to 0} \varepsilon^{-2} \SM_\SCC(B_{t + \varepsilon}(x; D_\SCC) \setminus B_t^\bullet(x; D_\SCC)). 
\end{equation*}
We call $Y_{-t}$ the \emph{boundary length} of the filled metric ball $B_t^\bullet(x; D_\SCC)$. The law of the process $\{Y_{-t}\}_{t \ge 0}$ is characterized as follows:
\begin{itemize}
    \item $Y_{-t} \to \infty$ as $t \to \infty$;
    \item for each deterministic $x > 0$, if $\tau_x \defeq \sup\{t \ge 0 : Y_{-t} = x\}$, then $\{Y_{s - \tau_x}\}_{s \in [0, \tau_x]}$ has the law of a $3/2$-stable CSBP started from $x$ and stopped upon hitting $0$.
\end{itemize}

For each deterministic $t > 0$, the boundary length $Y_{-t}$ of $B_t^\bullet(x; D_\SCC)$ is almost surely determined by
\begin{equation}\label{eq:filled-metric-ball-cone}
    \left(B_t^\bullet(x; D_\SCC), D_\SCC(\bullet, \bullet; B_t^\bullet(x; D_\SCC)), \SM_\SCC|_{B_t^\bullet(x; D_\SCC)}; x\right), 
\end{equation}
where $D_\SCC(\bullet, \bullet; B_t^\bullet(x; D_\SCC))$ denotes the internal metric, namely the infimum of the $D_\SCC$-lengths of paths contained in $B_t^\bullet(x; D_\SCC)$. The same boundary length is also almost surely determined by
\begin{equation}\label{eq:Brownian-horn}
    \left(\SCC \setminus B_t^\bullet(x; D_\SCC), D_\SCC(\bullet, \bullet; \SCC \setminus B_t^\bullet(x; D_\SCC)), \SM_\SCC|_{\SCC \setminus B_t^\bullet(x; D_\SCC)}\right). 
\end{equation}
Moreover, the objects in~\eqref{eq:filled-metric-ball-cone} and~\eqref{eq:Brownian-horn} are conditionally independent given $Y_{-t}$. The conditional law of~\eqref{eq:Brownian-horn} given $Y_{-t}$ does not depend on $t$. This law is called the law of a \emph{Brownian disk} with boundary length $Y_{-t}$ and infinite area; see \cite{MR4341082}.

The Brownian disk with infinite area satisfies a natural scaling rule. If $(\SCD, D_\SCD, \SM_\SCD)$ is a Brownian disk with boundary length $\ell$ and infinite area, then for every deterministic $\lambda > 0$, the rescaled triple $(\SCD, \lambda D_\SCD, \lambda^4\SM_\SCD)$ is a Brownian disk with boundary length $\lambda^2\ell$ and infinite area.

Let $0 \le s < t$. We define the \emph{metric band} by $A_{s,t}^\bullet(x; D_\SCC) \defeq B_t^\bullet(x; D_\SCC) \setminus B_s^\bullet(x; D_\SCC)$. The conditional law of
\begin{equation*}
    \left(A_{s,t}^\bullet(x; D_\SCC), D_\SCC(\bullet, \bullet; A_{s,t}^\bullet(x; D_\SCC)), \SM_\SCC|_{A_{s,t}^\bullet(x; D_\SCC)}; \partial B_s^\bullet(x; D_\SCC), \partial B_t^\bullet(x; D_\SCC)\right)
\end{equation*}
given $Y_{-s}$ and $Y_{-t}$ depends only on $Y_{-s}$, $Y_{-t}$, and the width $t - s$. Metric bands also satisfy the corresponding scaling property.

The Brownian sphere admits an analogous breadth-first exploration. We will not describe this exploration in detail, since it will not be needed in the present paper.

\subsection{Gaussian free field}

The Gaussian free field (GFF) is a centered Gaussian process. We can write it as a series expansion $\sum_{j = 1}^\infty X_j \psi_j$, where $(X_j)_{j \ge 1}$ is a sequence of independent standard Gaussian random variables, and $(\psi_j)_{j \ge 1}$ is an orthonormal basis for a corresponding Hilbert space using the Dirichlet inner product. Here are the two specific cases:
\begin{itemize}
    \item For a simply connected domain $U \subset \BC$, the \emph{zero-boundary GFF} $\Phi$ is a random Schwartz distribution on $U$, acting on test functions $f \in C_c^\infty(U)$. Its covariance function is
    \begin{equation*}
        \BE\lbrack\langle\Phi, f\rangle\langle\Phi, g\rangle\rbrack = \int_{\BC \times \BC} G_U(x, y) f(x) g(y) \, \rd x \, \rd y, 
    \end{equation*}
    where $G_U(\bullet, \bullet)$ is the Green function for $U$, normalized so that $G_U(x, y) \sim -\log(\lVert x - y\rVert)$ as $y \to x$. The Hilbert space $H_0^1(U)$ is the completion of $C_c^\infty(U)$ under the inner product $\langle f, g\rangle_{H_0^1(U)} \defeq \frac1{2\pi}\int_U \nabla f(z) \cdot \nabla g(z) \, \rd z$.
    \item The \emph{whole-plane GFF} $\Phi$ is defined on $\BC$ as a random Schwartz distribution modulo additive constants. The test functions must have mean zero, meaning $\{f \in C_c^\infty(\BC) : \int_\BC f(z) \, \rd z = 0\}$. Its covariance is
    \begin{equation*}
        \BE\lbrack\langle\Phi, f\rangle\langle\Phi, g\rangle\rbrack = -\int_{\BC \times \BC} \log(\lVert x - y\rVert) f(x) g(y) \, \rd x \, \rd y. 
    \end{equation*}
    Its Hilbert space $H^1(\BC)$ is the completion of $\{f \in C^\infty(\BC) / \BR : \int_\BC \|\nabla f(z)\|^2 \, \rd z < \infty\}$ under the inner product $\langle f, g\rangle_{H^1(\BC)} \defeq \frac1{2\pi}\int_\BC \nabla f(z) \cdot \nabla g(z) \, \rd z$.
\end{itemize}

The GFF satisfies the domain Markov property. Let $U \subset \BC$ be a deterministic simply connected domain, or let $U = \BC$. Let $\Phi$ be a zero-boundary GFF on $U$ in the first case, and a whole-plane GFF in the second case. Let $V \subset U$ be a simply connected subdomain. Let $\kH_V$ denote the harmonic extension to $V$ of $\Phi|_{U \setminus V}$. Then $\Phi - \kH_V$ is independent of $\kH_V$ and has the law of a zero-boundary GFF on $V$.

We observe that $H^1(\BC)$ has the orthogonal decomposition $H^1(\BC) = H^{\mathrm{rad}}(\BC) \oplus H^{\mathrm{circ}}(\BC)$, where
\begin{align*}
    H^{\mathrm{rad}}(\BC) &\defeq \left\{f \in H^1(\BC) : f \text{ is constant on } \{z \in \BC : |z| = t\} \text{ for each } t > 0\right\}; \\
    H^{\mathrm{circ}}(\BC) &\defeq \left\{f \in H^1(\BC) : f \text{ has mean zero on } \{z \in \BC : |z| = t\} \text{ for each } t > 0\right\}. 
\end{align*}
Let $\Phi$ be a whole-plane GFF. The projection of $\Phi$ onto $H^{\mathrm{rad}}(\BC)$ is called the \emph{circle average process} of $\Phi$. The process
\begin{equation*}
    \BR \to \BR \colon t \mapsto (\text{the circle average of } \Phi \text{ over } \partial B_{\re^{-t}}(0))
\end{equation*}
has the law of a standard two-sided Brownian motion modulo additive constant.

\subsection{Liouville quantum gravity}\label{subsection:background-LQG}

Let $\gamma \in (0, 2)$ and set $Q \defeq 2/\gamma + \gamma/2$. A \emph{$\gamma$-LQG surface} is an equivalence class of triples $(U,g,\Phi)$, where $U$ is a Riemann surface, $g$ is a conformal Riemannian metric on $U$ (i.e., a Riemannian metric of the form $\varrho(z)^2 \, \rd z \, \rd\overline z$ in each local chart), and $\Phi$ is a Schwartz distribution on $(U,g)$, typically given by an instance of the Gaussian free field. We declare two triples $(U_1,g_1,\Phi_1)$ and $(U_2,g_2,\Phi_2)$ to be equivalent if there is a conformal mapping $f \colon U_2 \to U_1$ and a smooth function $\varrho \colon U_2 \to \BR_{>0}$ such that $f^\ast(g_1) = \varrho^2 g_2$ and $\Phi_2 = \Phi_1 \circ f + Q\log(\varrho)$. A choice of representative $(U,g,\Phi)$ is called an \emph{embedding} of the $\gamma$-LQG surface.

We will also use marked $\gamma$-LQG surfaces. Suppose that $x_1,\dots,x_n$ are points in the completion of $U_1$, and that $y_1,\dots,y_n$ are points in the completion of $U_2$. Then $(U_1,g_1,\Phi_1; x_1,\dots,x_n)$ and $(U_2,g_2,\Phi_2; y_1,\dots,y_n)$ are equivalent if the above conformal mapping $f$ can be chosen so that $f(y_j) = x_j$ for every $j \in [1,n]_\BZ$. In the present paper, we will only consider embeddings for which $U$ is an open subset of $\BC$ and $g$ is the Euclidean metric, so we suppress $g$ from the notation.

As constructed in \cite{LQGKPZ}, a $\gamma$-LQG surface $(U,\Phi)$ carries a canonical random measure $\SM_\Phi$, called the \emph{$\gamma$-LQG area measure}. It is heuristically written as ``$\re^{\gamma\Phi} \, \rd x \, \rd y$'', and is invariant under changes of embedding in the sense described above.

There is also a canonical random metric $D_\Phi$ associated with $(U,\Phi)$, called the \emph{$\gamma$-LQG metric}. For $\gamma = \sqrt{8/3}$, this metric was constructed in \cite{MR3572845,MR4050102,MR4225028,MR4348679,MR4242633}. The construction for general $\gamma \in (0,2)$ was later obtained in \cite{TightLFPP,WeakLQGMet,ConfLQG,LocMetGFF,ExUniLQG}. Formally, $D_\Phi$ corresponds to the Riemannian metric ``$\re^{\gamma\Phi}(\rd x^2 + \rd y^2)$''. It is characterized by the following properties:
\begin{enumerate}[label=(\Roman*),ref=\Roman*]
    \item {\bfseries (Length metric)} $D_\Phi$ is almost surely a length metric that induces the Euclidean topology on $U$.
    \item {\bfseries (Locality)} For each deterministic open subset $V \subset U$, the internal metric $D_\Phi(\bullet, \bullet; V)$ is almost surely determined by $\Phi|_V$.
    \item {\bfseries (Weyl scaling)} There is a unique constant $\xi = \xi(\gamma) > 0$ (with $\xi = 1/\sqrt6$ when $\gamma = \sqrt{8/3}$) such that, almost surely, 
    \begin{equation*}
        D_{\Phi + \phi}(x, y) = \inf_{P \colon x \to y} \int_0^{\len(P; D_\Phi)} \re^{\xi\phi(P(t))} \, \rd t, \quad \forall x, y \in U, 
    \end{equation*}
    where the infimum is over all continuous paths $P$ from $x$ to $y$ parameterized by $D_\Phi$-length.
    \item {\bfseries (Coordinate change)} Let $f \colon \widetilde U \to U$ be a deterministic conformal mapping. Then, almost surely, 
    \begin{equation*}
        D_{\Phi \circ f + Q\log(\lVert f^\prime\rVert)}(x, y) = D_\Phi(f(x), f(y)), \quad \forall x, y \in \widetilde U. 
    \end{equation*}
\end{enumerate}

We next recall the definitions of the $\gamma$-LQG cone and the $\gamma$-LQG sphere; see \cite{dms2021mating}. Let $\{A_t\}_{t \in \BR}$ be the process defined as follows. For $t \ge 0$, set $A_t \defeq B_t + \gamma t$, where $\{B_t\}_{t \ge 0}$ is a standard Brownian motion started from $0$. For $t < 0$, set $A_t \defeq \widetilde B_{-t} + \gamma t$, where $\{\widetilde B_t\}_{t \ge 0}$ is an independent Brownian motion started from $0$ and conditioned so that $\widetilde B_t + (Q - \gamma)t > 0$ for every $t > 0$. Let $\Phi$ be a whole-plane GFF independent of $\{A_t\}_{t \in \BR}$. We define $\Phi^{\mathrm{cone}}$ to be the random Schwartz distribution whose projection onto $H^{\mathrm{rad}}(\BC)$ is given by $A_{-\log(|\bullet|)}$, and whose projection onto $H^{\mathrm{circ}}(\BC)$ agrees with that of $\Phi$. The pointed $\gamma$-LQG surface represented by $(\BC, \Phi^{\mathrm{cone}}; 0, \infty)$ is called a \emph{$\gamma$-LQG cone}. This choice of representative is called its \emph{circle average embedding}.

We now define the $\gamma$-LQG sphere. Let $e$ be sampled from the Bessel excursion measure of dimension $4 - 8/\gamma^2$. Let $\{A_t^e\}_{t \in \BR}$ be the process $(2/\gamma)\log(e)$, reparameterized so that it has quadratic variation $\rd t$. If, in the above construction of the $\gamma$-LQG cone, we replace $\{A_t\}_{t \in \BR}$ by $\{A_t^e\}_{t \in \BR}$, then we obtain the \emph{unconditioned two-pointed $\gamma$-LQG sphere measure}. This measure decomposes according to the total area as
\begin{equation*}
    C \int_0^\infty A^{-4/\gamma^2} \BP_{2,A} \, \rd A, 
\end{equation*}
where $C = C(\gamma) > 0$ is a constant and $\BP_{2,A}$ is the law of the two-pointed $\gamma$-LQG sphere conditioned to have area $A$.

Let $(\widehat\BC, \Phi; 0, \infty)$ be a two-pointed $\gamma$-LQG sphere with unit area. Conditional on $\Phi$, let $x$ and $y$ be independent samples from the $\gamma$-LQG area measure $\SM_\Phi$. Then $(\widehat\BC, \Phi; 0, \infty)$ and $(\widehat\BC, \Phi; x, y)$ agree in law as two-pointed $\gamma$-LQG surfaces. Therefore, for each $n \in \BN \cup \{0\}$, we define the \emph{$n$-pointed $\gamma$-LQG sphere with unit area} to be the $n$-pointed $\gamma$-LQG surface represented by $(\widehat\BC, \Phi; x_1, \dots, x_n)$, where $x_1, \dots, x_n$ are conditionally independent samples from $\SM_\Phi$.

The following absolute continuity statement will allow us to transfer several results from the whole-plane GFF to the $\gamma$-LQG sphere.

\begin{lemma}\label{lem:three-pointed-sphere}
    Let $(\widehat\BC, \Phi_1; 0, 1, \infty)$ be a three-pointed $\gamma$-LQG sphere with unit area. Let $\Phi_2$ be a whole-plane GFF. Then for each bounded open subset $U \Subset \BC \setminus \{0, 1\}$, the laws of $\Phi_1|_U$ and $\Phi_2|_U$, viewed as random Schwartz distributions modulo additive constants, are mutually absolutely continuous. 
\end{lemma}

\begin{proof}
    See, e.g., \cite[Lemma~2.4]{BorgaGwynneSun2025}. 
\end{proof}

It was established in \cite{MR4050102,MR4225028,MR4348679,MR4242633} that the Brownian sphere and the $\sqrt{8/3}$-LQG sphere agree in law as two-pointed metric measure spaces. Similarly, the Brownian plane and the $\sqrt{8/3}$-LQG cone agree in law as pointed metric measure spaces.

Let $\kappa \defeq \gamma^2 \in (0, 4)$ and $\kappa^\prime \defeq 16/\kappa > 4$. The whole-plane space-filling SLE$_{\kappa^\prime}$ curve from $\infty$ to $\infty$, first constructed in \cite{IG4}, provides a natural way to formulate the translation invariance property of the $\gamma$-LQG cone. We denote this curve by $\eta^\prime \colon \BR \to \BC$. It is almost surely non-self-crossing and space-filling, and satisfies $\eta^\prime(t) \to \infty$ as $t \to \pm\infty$. Up to reparameterization, the law of $\eta^\prime$ is invariant under M\"obius transformations that fix $\infty$. Thus, the whole-plane space-filling SLE$_{\kappa^\prime}$ curve is naturally defined on a Riemann sphere (hence a $\gamma$-LQG sphere or cone). We may choose the parameterization so that $\eta^\prime(0) = 0$. As shown in \cite{dms2021mating}, if $(\BC, \Phi; 0, \infty)$ is a $\gamma$-LQG cone independent of $\eta^\prime$, and if $\eta^\prime$ is parameterized by the $\gamma$-LQG area measure $\SM_\Phi$, then for each deterministic time $t \in \BR$, the pointed $\gamma$-LQG surface $(\BC, \Phi; \eta^\prime(t), \infty)$ also has the law of a $\gamma$-LQG cone.

\begin{lemma}\label{lem:net}
    Let $(\BC, \Phi; 0, \infty)$ be a $\sqrt{8/3}$-LQG cone. Let $\eta^\prime \colon \BR \to \BC$ be an independent whole-plane space-filling SLE$_6$ curve from $\infty$ to $\infty$, parameterized by the LQG area measure $\SM_\Phi$. Fix $T > 0$. Given $\Phi$ and $\eta^\prime$, let $\{x_j\}_{j \in \BN}$ be conditionally independent samples from $\SM_\Phi|_{\eta^\prime([-T, T])}$. Then for each $\zeta > 0$, almost surely on the event that $B_1^\bullet(0; D_\Phi) \subset \eta^\prime([-T, T])$, there exists $\varepsilon_\ast \in (0, 1)$ such that
    \begin{equation*}
        B_1^\bullet(0; D_\Phi) \subset \bigcup \left\{B_\varepsilon(x_j; D_\Phi) : j \in [1, \varepsilon^{-4 - \zeta}]_\BZ, \ x_j \in B_1^\bullet(0; D_\Phi)\right\}, \quad \forall \varepsilon \in (0, \varepsilon_\ast].
    \end{equation*}
\end{lemma}

\begin{proof}
    Let $X$ and $Z$ be as in the definition of the Brownian plane. By using Kolmogorov's continuity criterion, one verifies immediately that $Z$ is almost surely locally $\alpha$-H\"older continuous of every order $\alpha \in (0, 1/4)$. This implies that, almost surely, for each compact subset $K \subset \BC$, there exists $\varepsilon_\ast \in (0, 1)$ such that 
    \begin{equation*}
        \SM_\Phi(B_\varepsilon(x; D_\Phi)) \ge \varepsilon^{4 + \zeta}, \quad \forall x \in K, \ \forall \varepsilon \in (0, \varepsilon_\ast]. 
    \end{equation*}
    Combining this with a standard argument, we complete the proof. 
\end{proof}

\subsection{Independence for the Brownian plane and the Gaussian free field}\label{subsection:independence}

In the present subsection, we record several results which can be summarized as follows:
\begin{itemize}
    \item Disjoint concentric metric bands of a Brownian plane are ``nearly independent''.
    \item The restrictions of a whole-plane GFF to disjoint concentric Euclidean annuli, or to disjoint uniformly separated Euclidean balls, are ``nearly independent'' when viewed modulo additive constants.
\end{itemize}
These properties make the Brownian plane and the whole-plane GFF particularly convenient to work with.

\begin{lemma}\label{lem:independence-across-bands}
    For each $\lambda \in (0, 1)$, $\alpha > 0$, and $b \in (0, 1)$, there exists $p = p(\lambda, \alpha, b) \in (0, 1)$ and $C = C(\lambda, \alpha, b) > 0$ such that the following is true: Let $(\BC, \Phi; 0, \infty)$ be a $\sqrt{8/3}$-LQG cone. Let $\{t_j\}_{j \in \BN}$ be a decreasing sequence of positive real numbers such that $t_{j + 1}/t_j \le \lambda$ for every $j \in \BN$. Let $\{E_j\}_{j \in \BN}$ be a sequence of events such that each $E_j$ is almost surely determined by the $\sqrt{8/3}$-LQG surface parameterized by $A_{t_{j + 1},t_j}^\bullet(0; D_\Phi)$. Suppose that $\BP\lbrack E_j\rbrack \ge p$ for every $j \in \BN$. Then for each $n \in \BN$, it holds with probability at least $1 - C\re^{-\alpha n}$ that there are at least $bn$ values of $j \in [1,n]_\BZ$ for which $E_j$ occurs.
\end{lemma}

\begin{proof}
    See \cite[Remark~3.12]{2026arXiv260324473M}. 
\end{proof}

\begin{lemma}\label{lem:independence-across-scales}
    For each $\lambda \in (0, 1)$, $\alpha > 0$, and $b \in (0, 1)$, there exists $p = p(\lambda, \alpha, b) \in (0, 1)$ and $C = C(\lambda, \alpha, b) > 0$ such that the following is true: Let $\Phi$ be a whole-plane GFF. Let $\{t_j\}_{j \in \BN}$ be a decreasing sequence of positive real numbers such that $t_{j + 1}/t_j \le \lambda$ for every $j \in \BN$. Let $\{E_j\}_{j \in \BN}$ be a sequence of events such that each $E_j$ is almost surely determined by $\Phi|_{A_{t_{j + 1},t_j}(0)}$. Suppose that $\BP\lbrack E_j\rbrack \ge p$ for every $j \in \BN$. Then for each $n \in \BN$, it holds with probability at least $1 - C\re^{-\alpha n}$ that there are at least $bn$ values of $j \in [1,n]_\BZ$ for which $E_j$ occurs. (Here, both $\Phi$ and $\Phi|_{A_{t_{j + 1},t_j}(0)}$ are viewed as random Schwartz distributions modulo additive constants.)
\end{lemma}

\begin{proof}
    See \cite[Lemma~3.1]{LocMetGFF}. 
\end{proof}

\begin{lemma}\label{lem:independence-across-balls}
    For each $p, q \in (0, 1)$ and $s > 0$, there exists $n = n(p, q, s) \in \BN$ such that the following is true: Let $\Phi$ be a whole-plane GFF. Let $r > 0$. Let $z_1, \dots, z_n \in \BC$ such that $\lVert z_j - z_k\rVert \ge 2(1 + s)r$ for all distinct $j, k \in [1, n]_\BZ$. Let $E_1, \dots, E_n$ be a sequence of events such that each $E_j$ is almost surely determined by $\Phi|_{B_r(z_j)}$. Suppose that $\BP\lbrack E_j\rbrack \ge p$ for every $j \in [1, n]_\BZ$. Then $\BP\lbrack E_1 \cup \cdots \cup E_n\rbrack \ge q$. (Here, both $\Phi$ and $\Phi|_{B_r(z_j)}$ are viewed as random Schwartz distributions modulo additive constants.)
\end{lemma}

\begin{proof}
    See \cite[Lemma~2.7]{ExUniLQG}. 
\end{proof}

\subsection{Quasiconformal mappings}\label{subsection:background-quasiconformal}

In the present subsection, we review several equivalent definitions of quasiconformal mappings and recall some basic properties; see, e.g., \cite{MR2241787,MR344463,MR454009} for more details.

Let $U \subset \BC$ be a planar domain and let $f \colon U \to f(U) \subset \BC$ be a homeomorphism. We say that $f$ is \emph{($K$-)quasiconformal} if there is a constant $K \ge 1$ such that any, and hence all, of the following equivalent conditions are satisfied:
\begin{enumerate}[label=(\Roman*),ref=\Roman*]
    \item\label{it:def-quasiconformal-1} $f \in W_{\mathrm{loc}}^{1,2}(U, \BC)$, and $\lVert\overline\partial f(z)\rVert \le \left(\frac{K - 1}{K + 1}\right)\lVert\partial f(z)\rVert$ for almost every $z \in U$.
    \item\label{it:def-quasiconformal-2} There is a constant $H \ge 1$ such that
    \begin{equation*}
        \limsup_{r \to 0} \frac{\sup\{\lVert f(x) - f(y)\rVert : y \in \partial B_r(x)\}}{\inf\{\lVert f(x) - f(y)\rVert : y \in \partial B_r(x)\}} \le H, \quad \forall x \in U. 
    \end{equation*}
    \item\label{it:def-quasiconformal-3} There is a constant $H \geq 1$ such that
    \begin{equation*}
        \liminf_{r \to 0} \frac{\sup\{\lVert f(x) - f(y)\rVert : y \in \partial B_r(x)\}}{\inf\{\lVert f(x) - f(y)\rVert : y \in \partial B_r(x)\}} \le H, \quad \forall x \in U. 
    \end{equation*}
    \item\label{it:def-quasiconformal-4} We have $K^{-1}\Mod(\Gamma) \le \Mod(f(\Gamma)) \le K\Mod(\Gamma)$ for every family $\Gamma$ of paths in $U$, where $\Mod(\Gamma)$ denotes the conformal modulus of $\Gamma$.
    \item\label{it:def-quasiconformal-5} We have $K^{-1}\Mod(Q) \le \Mod(f(Q)) \le K\Mod(Q)$ for every quadrilateral $Q \subset U$, where $\Mod(Q)$ denotes the conformal modulus of $Q$. (Recall that a quadrilateral is a Jordan domain with four marked boundary points.)
\end{enumerate}
Here, the constants $H$ and $K$ depend only on each other. The equivalence of conditions~\eqref{it:def-quasiconformal-1},~\eqref{it:def-quasiconformal-2},~\eqref{it:def-quasiconformal-4}, and~\eqref{it:def-quasiconformal-5} is a deep result due to the work of Gehring, V\"ais\"al\"a, and many others in the 1950s and early 1960s. The equivalence of condition~\eqref{it:def-quasiconformal-3} with the other conditions, which is particularly important for the purposes of the present paper, follows from \cite{MR1323982,MR1771571}. Condition~\eqref{it:def-quasiconformal-1} is usually referred to as the \emph{analytic definition} of quasiconformal mappings; conditions~\eqref{it:def-quasiconformal-2} and~\eqref{it:def-quasiconformal-3} are referred to as the \emph{metric definitions}; and conditions~\eqref{it:def-quasiconformal-4} and~\eqref{it:def-quasiconformal-5} are referred to as the \emph{geometric definitions}. 

Note that $K$-quasiconformality is a local property: If $f$ is $K$-quasiconformal in an open neighborhood of each point of $U$, then $f$ is $K$-quasiconformal in $U$.

The inverse of a $K$-quasiconformal mapping is also $K$-quasiconformal. The composition of a $K_1$-quasiconformal mapping and a $K_2$-quasiconformal mapping is $K_1K_2$-quasiconformal. A mapping is conformal if and only if it is $1$-quasiconformal. 

The \emph{measurable Riemann mapping theorem}, proved in \cite{MR115006}, says that a converse to condition~\eqref{it:def-quasiconformal-1} holds: For each $\mu \in L^\infty(U, \BC)$ with $\lVert\mu\rVert_\infty < 1$, there exists a quasiconformal mapping $f \colon U \to f(U)$ satisfying the Beltrami equation $\overline\partial f = \mu \, \partial f$.

There is no quasiconformal mapping from $\BC$ onto a proper subdomain $U \subsetneq \BC$. Moreover, a mapping from $\BC$ to $\BC$ is quasiconformal if and only if it is quasisymmetric with respect to the Euclidean metric.

\section{Brownian quasisymmetry implies Euclidean quasiconformality}
\label{section:quasiconformality}

In the present section, we prove the following intermediate result, which will be used in the proof of the main results.

We call a mapping a \emph{quasisymmetric embedding} if it is a quasisymmetric homeomorphism onto its image.

\begin{proposition}\label{lem:quasiconformality}
    For each distortion function $\eta$, there exists $K = K(\eta) \ge 1$ such that the following is true: Let $(\BC, \Phi_1; 0, \infty)$ and $(\BC, \Phi_2; 0, \infty)$ be a pair of coupled (not necessarily independent) $\sqrt{8/3}$-LQG cones. Then, almost surely, every $\eta$-quasisymmetric embedding of $(B_1^\bullet(0; D_{\Phi_1}), D_{\Phi_1})$ into $(\BC, D_{\Phi_2})$ is $K$-quasiconformal.
\end{proposition}

The idea of the proof of \Cref{lem:quasiconformality} is as follows. First, by choosing a large ratio for the metric bands of the first $\sqrt{8/3}$-LQG surface and applying independence across metric bands, we may arrange that, with very high probability, the metric bands at a large portion of scales have conformal moduli bounded below by a prescribed constant; see \Cref{lem:independence-across-bands-moduli}. We then choose this prescribed constant sufficiently large so that whenever a metric band has conformal modulus at least this constant, it must contain a Euclidean circle centered at its center; see \Cref{lem:Teichmueller}.

Next, we choose an even larger ratio for the metric bands of the second $\sqrt{8/3}$-LQG surface, in a manner depending on the distortion function, so that the image of each metric band of the first $\sqrt{8/3}$-LQG surface with the ratio discussed above is contained in a metric band of the second $\sqrt{8/3}$-LQG surface with this larger ratio. Applying independence across metric bands again (cf.~\Cref{lem:independence-across-bands}), we may arrange that, with very high probability, a large portion of these metric bands of the second $\sqrt{8/3}$-LQG surface have conformal moduli bounded above and below. Combining this with a general lemma on the conformal moduli of doubly connected domains (cf.~\Cref{lem:Teichmueller}), we can then conclude that there exist arbitrarily small Euclidean circles as above whose images are contained in Euclidean annuli with bounded ratio. By condition~\eqref{it:def-quasiconformal-3}, this implies that the mapping is quasiconformal.

\subsection{Conformal moduli of metric bands}

In the present subsection, we record several estimates concerning the conformal moduli of metric bands of a $\sqrt{8/3}$-LQG cone. These results are essentially contained in the proofs in \cite{2026arXiv260324473M}.

\begin{proposition}\label{lem:moment}
    \begin{enumerate}
        \item\label{it:moment-0} Let $(\BC, \Phi; 0, \infty)$ be a $\sqrt{8/3}$-LQG cone. Then for each $p > 0$, 
        \begin{equation*}
            \BE\!\left\lbrack\re^{-2\pi p\Mod(A_{\varepsilon,1}^\bullet(0; D_\Phi))}\right\rbrack = O(\varepsilon^q) \quad \text{as } \varepsilon \to 0, \quad \forall q \in (0, \sqrt{1 + 12p} - 1). 
        \end{equation*}
        (Here, $\sqrt{1 + 12p} - 1$ comes from $(\sqrt{(Q - \gamma)^2 + 2p} - (Q - \gamma))/\xi$ when $\gamma = \sqrt{8/3}$, using the notation of \Cref{subsection:background-LQG}.) In particular, for each $M > 0$, it holds with superpolynomially high probability as $\varepsilon \to 0$ that $\Mod(A_{\varepsilon,1}^\bullet(0; D_\Phi)) \ge M$. 
        \item\label{it:moment-1} Let $(\SCD, D_\SCD, \SM_\SCD)$ be a Brownian disk with unit boundary length and infinite area. Then for each $p > 0$, 
        \begin{equation*}
            \BE\!\left\lbrack\re^{-2\pi p\Mod(B_T^\bullet(\partial\SCD; D_\SCD))}\right\rbrack = O(T^{-q}) \quad \text{as } T \to \infty, \quad \forall q \in (0, \sqrt{1 + 12p} - 1). 
        \end{equation*}
        In particular, for each $M > 0$, it holds with superpolynomially high probability as $T \to \infty$ that $\Mod(B_T^\bullet(\partial\SCD; D_\SCD)) \ge M$. 
    \end{enumerate}
\end{proposition}

\begin{proof}
    This follows immediately from a similar argument to the argument applied in the proof of \cite[Proposition~3.6 and Proposition~3.7]{2026arXiv260324473M}. 
\end{proof}

\begin{proposition}\label{lem:independence-across-bands-moduli}
    For each $\alpha > 0$, $b \in (0, 1)$, and $M > 0$, there exists $\lambda_\ast = \lambda_\ast(\alpha, b, M) \in (0, 1)$ such that for each $\lambda \in (0, \lambda_\ast]$, the following is true: Let $(\BC, \Phi; 0, \infty)$ be a $\sqrt{8/3}$-LQG cone. Let $\{t_j\}_{j \in \BN}$ be a decreasing sequence of positive real numbers such that $t_{j + 1}/t_j \le \lambda$ for every $j \in \BN$. Then for each $n \in \BN$, it holds with probability at least $1 - \lambda^{\alpha n}$ that there are at least $bn$ values of $j \in [1, n]_\BZ$ for which $\Mod(A_{\lambda t_j,t_j}^\bullet(0; D_\Phi)) \ge M$. 
\end{proposition}

This follows immediately from \Cref{lem:moment}, together with an argument similar to the one applied in the proof of \Cref{lem:independence-across-bands}. For completeness, we include a proof.

\begin{lemma}\label{lem:independence-across-bands-boundary-lengths}
    For each $\zeta > 0$ and $b \in (0, 1)$, there exists $\lambda_\ast = \lambda_\ast(\zeta, b) \in (0, 1)$ such that for each $\lambda \in (0, \lambda_\ast]$, the following is true: Let $(\BC, \Phi; 0, \infty)$ be a $\sqrt{8/3}$-LQG cone. Let $\{t_j\}_{j \in \BN}$ be a decreasing sequence of positive real numbers such that $t_{j + 1}/t_j \le \lambda$ for all $j \in \BN$. Then for each $n \in \BN$, it holds with probability at least $1 - \exp(-\lambda^{-\zeta}n)$ that there are at least $bn$ values of $j \in [1, n]_\BZ$ for which the boundary length of the filled metric ball $B_{t_j}^\bullet(0; D_\Phi)$ is at most $\lambda^{-2\zeta} t_j^2$. 
\end{lemma}

\begin{proof}
    See \cite[Lemma~3.2,~(ii)]{2026arXiv260324473M}. 
\end{proof}

\begin{lemma}\label{lem:independence-across-bands-moduli-proof}
    For each $\alpha > 0$ and $M > 0$, there exists $\lambda_\ast = \lambda_\ast(\alpha, M) \in (0,1)$ such that, for each $\lambda \in (0,\lambda_\ast]$, the following is true. Let $(\BC, \Phi; 0, \infty)$ be a $\sqrt{8/3}$-LQG cone. Let $\{t_j\}_{j \in \BN}$ be a decreasing sequence of positive real numbers such that $t_{j + 1}/t_j \le \lambda$ for every $j \in \BN$. Then, for each finite subset $J \subset \BN$, with probability at least $1 - (\lambda^{\alpha})^{\#J}$, there exists $j \in J$ for which $\Mod(A_{\lambda t_j,t_j}^\bullet(0; D_\Phi)) \ge M$. 
\end{lemma}

\begin{proof}[Proof of \Cref{lem:independence-across-bands-moduli} assuming \Cref{lem:independence-across-bands-moduli-proof}]
    By \Cref{lem:independence-across-bands-moduli-proof}, for each $J \subset [1, n]_\BZ$ with $\#J = \lceil(1 - b)n\rceil$, the probability that there does not exist $j \in J$ for which $\Mod(A_{\lambda t_j,t_j}^\bullet(0; D_\Phi)) \ge M$ is at most $\lambda^{\alpha(1 - b)n}$. Thus, by a union bound, 
    \begin{align*}
        &\BP\!\left\lbrack\text{there are fewer than } bn \text{ values of } j \in [1,n]_\BZ \text{ for which } \Mod(A_{\lambda t_j,t_j}^\bullet(0; D_\Phi)) \ge M\right\rbrack \\
        &\le \sum_{J \subset [1, n]_\BZ : \#J = \lceil(1 - b)n\rceil} \BP\!\left\lbrack\text{there does not exist } j \in J \text{ for which } \Mod(A_{\lambda t_j,t_j}^\bullet(0; D_\Phi)) \ge M\right\rbrack \\
        &\le \binom n{\lceil(1 - b)n\rceil} \lambda^{\alpha(1 - b)n} \le 2^n \lambda^{\alpha(1 - b)n}. 
    \end{align*}
    By increasing $\alpha$, we obtain the desired assertion.
\end{proof}

\begin{proof}[Proof of \Cref{lem:independence-across-bands-moduli-proof}]
    By replacing $\{t_j\}_{j \in \BN}$ with a subsequence, we may assume without loss of generality that $J = [1, n]_\BZ$. Moreover, by replacing $\{t_j\}_{j \in \BN}$ with $\{\lambda t_j\}_{j \in \BN}$, it suffices to show that for each $n \in \BN$, it holds with probability at least $1 - \lambda^{\alpha n}$ that there exists $j \in [1, n]_\BZ$ for which $\Mod(A_{t_j,\lambda^{-1} t_j}^\bullet(0; D_\Phi)) \ge M$.
    
    By \Cref{lem:moment},~\eqref{it:moment-1}, there exists $T_\ast > 0$ such that the following holds. Let $(\SCD, D_\SCD, \SM_\SCD)$ be a Brownian disk with unit boundary length and infinite area. Then, for each $T \ge T_\ast$, with probability at least $1 - T^{-\alpha}$, we have $\Mod(B_T^\bullet(\partial\SCD; D_\SCD)) \ge M$.

    By \Cref{lem:independence-across-bands-boundary-lengths}, there exists $\lambda_0 \in (0,1)$ such that, for each $\lambda \in (0,\lambda_0]$, the following holds. Let $\{t_j\}_{j \in \BN}$ be a decreasing sequence of positive real numbers such that $t_{j + 1}/t_j \le \lambda$ for every $j \in \BN$. Then, for each $n \in \BN$, with probability at least $1 - \exp(-\lambda^{-1/2}n)$, there are at least $n/2$ values of $j \in [1,n]_\BZ$ for which the boundary length of the filled metric ball $B_{t_j}^\bullet(0; D_\Phi)$ is at most $\lambda^{-1} t_j^2$.

    Set $\lambda_\ast \defeq \lambda_0 \wedge (2T_\ast)^{-2}$. Fix $\lambda \in (0,\lambda_\ast]$. Let $\{t_j\}_{j \in \BN}$ be a decreasing sequence of positive real numbers such that $t_{j + 1}/t_j \le \lambda$ for every $j \in \BN$. Fix $n \in \BN$. We now perform a forward metric exploration using filled metric balls. To simplify notation, for each $j \in [1,n]_\BZ$, write $s_j \defeq t_{n + 1 - j}$. For each $k \in \BN$, let $j_k$ be the $k$-th smallest $j \in [1,n]_\BZ$ for which the boundary length of $B_{s_j}^\bullet(0; D_\Phi)$ is at most $\lambda^{-1} s_j^2$, with the convention that $j_k = \infty$ if there are fewer than $k$ such values of $j$. The preceding paragraph implies that $\BP\lbrack j_{\lceil n/2\rceil} \le n\rbrack \ge 1 - \exp(-\lambda^{-1/2}n)$.

    On the other hand, for each $k \in \BN$, by the Markov property of the forward metric exploration of the Brownian plane, on the event that $j_k < \infty$, conditional on the $\sqrt{8/3}$-LQG surface parameterized by $B_{s_{j_k}}^\bullet(0; D_\Phi)$, the $\sqrt{8/3}$-LQG surface parameterized by $\BC \setminus B_{s_{j_k}}^\bullet(0; D_\Phi)$ is a Brownian disk with boundary length given by the boundary length of $B_{s_{j_k}}^\bullet(0; D_\Phi)$ and infinite area. Thus, by the scaling property of the Brownian disk with infinite area,
    \begin{align*}
        \BP\!\left\lbrack\Mod(A_{s_{j_k},\lambda^{-1} s_{j_k}}^\bullet(0; D_\Phi)) \ge M \ \middle\vert \ B_{s_{j_k}}^\bullet(0; D_\Phi)\right\rbrack &= \BP\!\left\lbrack\Mod(B_{(\lambda^{-1} - 1)s_{j_k}Y_{-s_{j_k}}^{-1/2}}^\bullet(\partial\SCD; D_\SCD)) \ge M\right\rbrack \\
        &\ge \BP\!\left\lbrack\Mod(B_{\lambda^{-1/2}/2}^\bullet(\partial\SCD; D_\SCD)) \ge M\right\rbrack \\
        &\ge 1 - (2\lambda^{1/2})^\alpha, 
    \end{align*}
    where $Y_{-s_{j_k}}$ denotes the boundary length of $B_{s_{j_k}}^\bullet(0; D_\Phi)$.

    Therefore,
    \begin{align*}
        &\BP\!\left\lbrack\text{there does not exist } j \in [1, n]_\BZ \text{ such that } \Mod(A_{t_j,\lambda^{-1}t_j}^\bullet(0; D_\Phi)) \ge M\right\rbrack \\
        &\le \BP\lbrack j_{\lceil n/2\rceil} = \infty\rbrack + \BP\!\left\lbrack\text{there does not exist } k \in [1, n/2]_\BZ \text{ such that } \Mod(A_{s_{j_k},\lambda^{-1}s_{j_k}}^\bullet(0; D_\Phi)) \ge M\right\rbrack \\
        &\le \exp(-\lambda^{-1/2}n) + (2\lambda^{1/2})^{\alpha n/2}.
    \end{align*}
    Choosing $\alpha > 0$ sufficiently large and $\lambda_\ast > 0$ sufficiently small, we complete the proof.
\end{proof}

\subsection{Proof of \Cref{lem:quasiconformality}}

\begin{lemma}\label{lem:Teichmueller}
    Let $A \subset \BC$ be a doubly connected domain. Let $z$ belong to the bounded connected component of $\BC \setminus A$. Set
    \begin{align*}
        R_1 &\defeq \inf\{R > 0 : B_R(z) \text{ contains the bounded connected component of } \BC \setminus A\}; \\
        R_2 &\defeq \sup\{R > 0 : B_R(z) \text{ does not intersect the unbounded connected component of } \BC \setminus A\}. 
    \end{align*}
    Then 
    \begin{equation*}
        \left(\re^{2\pi\Mod(A)}/16 - 1\right) \vee \left(16\re^{-\frac\pi{2\Mod(A)}}\right) \le R_2/R_1 \le \re^{2\pi\Mod(A)}.
    \end{equation*}
    In particular, if $\Mod(A) > \frac{5\log(2)}{2\pi}$, then there exists $R > 0$ such that $\partial B_R(z) \subset A$. 
\end{lemma}

\begin{proof}
    We may assume without loss of generality that $z = 0$. If $R_2 \le R_1$, then $\re^{2\pi\Mod(A)} \ge 1 \ge R_2/R_1$; otherwise, $\Mod(A) \ge \Mod(A_{R_1,R_2}(0)) = \frac1{2\pi}\log(R_2/R_1)$. On the other hand, by \cite[Theorem~4-7]{CoInAhl}, $\Mod(A) \le \Mod(\BC \setminus ([-R_1, 0] \cup [R_2, \infty)))$. By \cite[(4-21)]{CoInAhl} as well as the discussion below it, $\Mod(\BC \setminus ([-R_1, 0] \cup [R_2, \infty))) \le \frac1{2\pi}\log(16(R_2/R_1 + 1))$ and
    \begin{equation*}
        \Mod(\BC \setminus ([-R_1, 0] \cup [R_2, \infty))) = \frac1{4\Mod(\BC \setminus ([-R_2, 0] \cup [R_1, \infty)))} \le \frac{2\pi}{4\log(16R_1/R_2)}. 
    \end{equation*}
    This completes the proof. 
\end{proof}

\begin{lemma}\label{lem:quasiconformality-proof}
    Let $(\BC, \Phi; 0, \infty)$ be a $\sqrt{8/3}$-LQG cone. Let $\eta^\prime \colon \BR \to \BC$ be an independent whole-plane space-filling SLE$_{\kappa^\prime}$ curve from $\infty$ to $\infty$, parameterized by the LQG area measure $\SM_\Phi$. Fix $T > 0$ and $\varepsilon_\ast \in (0,1)$. Given $\Phi$ and $\eta^\prime$, let $\{x_j\}_{j \in \BN}$ be conditionally independent samples from $\SM_\Phi|_{\eta^\prime([-T, T])}$. Write $E(T,\varepsilon_\ast)$ for the event that $B_1^\bullet(0; D_\Phi) \subset \eta^\prime([-T, T])$ and
    \begin{equation*}
        B_1^\bullet(0; D_\Phi) \subset \bigcup \left\{B_\varepsilon(x_j; D_\Phi) : j \in [1, \varepsilon^{-5}]_\BZ, \ x_j \in B_1^\bullet(0; D_\Phi)\right\}, \quad \forall \varepsilon \in (0, \varepsilon_\ast], 
    \end{equation*}
    Then:
    \begin{enumerate}
        \item\label{it:quasiconformality-proof-0} For each $M > 0$, there exists $\lambda = \lambda(M) \in (0, 1)$ such that, on the event $E(T,\varepsilon_\ast)$ and outside an event with polynomially small probability as $\varepsilon \to 0$, for every $z \in B_1^\bullet(0; D_\Phi)$, there are at least $\frac12\log_{1/\lambda}(1/\varepsilon)$ values of $t \in [\varepsilon^2, \varepsilon] \cap \{\lambda^n\}_{n \in \BZ}$ for which $\Mod(A_{\lambda t,t}^\bullet(z; D_\Phi)) \ge M$.
        \item\label{it:quasiconformality-proof-1} For each $\lambda, b \in (0, 1)$, there exists $M = M(\lambda, b) > 0$ such that, on the event $E(T,\varepsilon_\ast)$ and outside an event with polynomially small probability as $\varepsilon \to 0$, for every $z \in B_1^\bullet(0; D_\Phi)$, there are at least $b\log_{1/\lambda}(1/\varepsilon)$ values of $t \in [\varepsilon^2, \varepsilon] \cap \{\lambda^n\}_{n \in \BZ}$ for which
        \begin{equation*}
            1/M \le \Mod(A_{\lambda t,t}^\bullet(z; D_\Phi)) \le M. 
        \end{equation*}
    \end{enumerate}
\end{lemma}

\begin{proof}
    Recall that each $(\BC, \Phi; x_j, \infty)$ has the law of a $\sqrt{8/3}$-LQG cone.
    
    First, we consider assertion~\eqref{it:quasiconformality-proof-0}. By \Cref{lem:independence-across-bands-moduli} and a union bound, there exists $\lambda \in (0, 1)$ such that, with polynomially high probability as $\varepsilon \to 0$, for each $j \in [1, \varepsilon^{-15}]_\BZ$, there are at least $\frac12\log_{1/\lambda}(1/\varepsilon)$ values of $t \in [\varepsilon^2, \varepsilon] \cap \{\lambda^n\}_{n \in \BZ}$ for which $\Mod(A_{2\lambda t,t/2}^\bullet(x_j; D_\Phi)) \ge M$. On the other hand, on the event $E(T,\varepsilon_\ast)$, for each sufficiently small $\varepsilon \in (0,1)$ and each $z \in B_1^\bullet(0; D_\Phi)$, there exists $j \in [1, \varepsilon^{-15}]_\BZ$ such that $D_\Phi(z,x_j) \le \varepsilon^3$, in which case $A_{2\lambda t,t/2}^\bullet(x_j; D_\Phi) \subset A_{\lambda t,t}^\bullet(z; D_\Phi)$ for all $t \in [\varepsilon^2, \varepsilon] \cap \{\lambda^n\}_{n \in \BZ}$. This completes the proof of assertion~\eqref{it:quasiconformality-proof-0}.

    Next, we consider assertion~\eqref{it:quasiconformality-proof-1}. By \Cref{lem:independence-across-bands} and a union bound, there exists $M = M(\lambda, b) > 0$ such that, with polynomially high probability as $\varepsilon \to 0$, for each $j \in [1, \varepsilon^{-15}]_\BZ$, there are at least $b\log_{1/\lambda}(1/\varepsilon)$ values of $t \in [\varepsilon^2, \varepsilon] \cap \{\lambda^n\}_{n \in \BZ}$ for which $1/M \le \Mod(A_{\lambda t/2,2t}^\bullet(x_j; D_\Phi)) \le M$. On the other hand, on the event $E(T,\varepsilon_\ast)$, for each sufficiently small $\varepsilon \in (0,1)$ and each $z \in B_1^\bullet(0; D_\Phi)$, there exists $j \in [1, \varepsilon^{-15}]_\BZ$ such that $D_\Phi(z,x_j) \le \varepsilon^3$, in which case $A_{\lambda t,t}^\bullet(z; D_\Phi) \subset A_{\lambda t/2,2t}^\bullet(x_j; D_\Phi)$ for all $t \in [\varepsilon^2, \varepsilon] \cap \{\lambda^n\}_{n \in \BZ}$. This completes the proof of assertion~\eqref{it:quasiconformality-proof-1}.
\end{proof}

\begin{figure}[ht!]
    \centering
    \includegraphics[width=\linewidth]{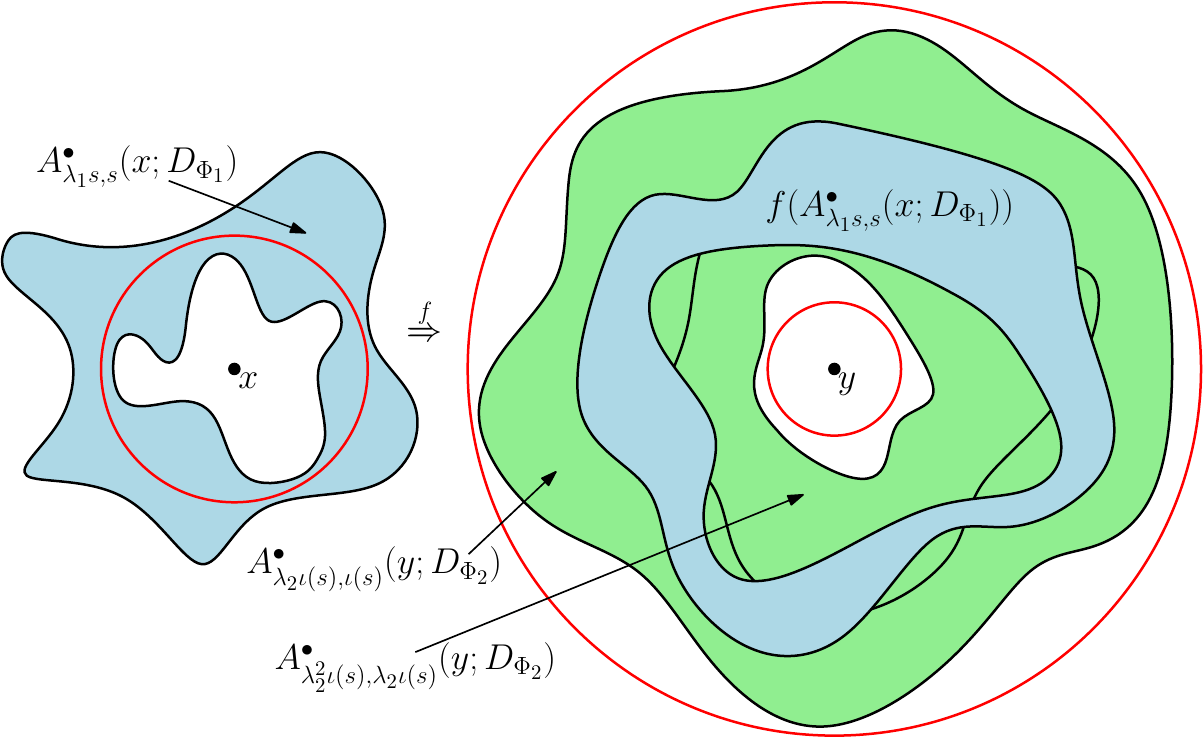}
    \caption{Illustration of the proof of \Cref{lem:quasiconformality}. {\bfseries Left:} The blue region depicts the metric band $A_{\lambda_1s,s}^\bullet(x; D_{\Phi_1})$ of the first $\sqrt{8/3}$-LQG cone $(\BC, \Phi_1; 0, \infty)$. The parameter $\lambda_1$ is chosen to be sufficiently small so that, with very high probability, the conformal modulus of $A_{\lambda_1s,s}^\bullet(x; D_{\Phi_1})$ is bounded below by the constant $\frac{5\log(2)}{2\pi}$ for a large proportion of values of $s \in \{\lambda_1^n\}_{n \in \BZ}$. The condition $\Mod(A_{\lambda_1s,s}^\bullet(x; D_{\Phi_1})) > \frac{5\log(2)}{2\pi}$ implies that this metric band contains a Euclidean circle centered at $x$. {\bfseries Right:} Set $y = f(x)$ and $\lambda_2 = 1/\eta(1/\lambda_1)$. The green regions depict the metric bands $A_{\lambda_2\iota(s),\iota(s)}^\bullet(y; D_{\Phi_2})$ and $A_{\lambda_2^2\iota(s),\lambda_2\iota(s)}^\bullet(y; D_{\Phi_2})$, with the image of $A_{\lambda_1s,s}^\bullet(x; D_{\Phi_1})$ overlaid on them. Here $\iota(s)$ is defined to be the largest $t \in \{\lambda_2^n\}_{n \in \BZ}$ such that $f(A_{\lambda_1s,s}^\bullet(x; D_{\Phi_1})) \cap A_{\lambda_2t,t}^\bullet(y; D_{\Phi_2}) \neq \emptyset$. The definitions of $\eta$-quasisymmetry and $\lambda_2$ imply that $f(A_{\lambda_1s,s}^\bullet(x; D_{\Phi_1})) \subset A_{\lambda_2^2\iota(s),\iota(s)}^\bullet(y; D_{\Phi_2})$. By independence across metric bands, we may arrange that, for a large proportion of values of $t \in \{\lambda_2^n\}_{n \in \BZ}$, the conformal modulus of $A_{\lambda_2t,t}^\bullet(y; D_{\Phi_2})$ is bounded above and below. The pigeonhole principle then implies that there exist arbitrarily small values of $s \in \{\lambda_1^n\}_{n \in \BZ}$ for which $\Mod(A_{\lambda_1s,s}^\bullet(x; D_{\Phi_1})) > \frac{5\log(2)}{2\pi}$ and $\Mod(A_{\lambda_2t,t}^\bullet(y; D_{\Phi_2}))$ is bounded above and below for each of the four values $t \in \{\lambda_2^2\iota(s), \lambda_2\iota(s), \iota(s), \lambda_2^{-1}\iota(s)\}$. The latter condition implies that $A_{\lambda_2^2\iota(s),\iota(s)}^\bullet(y; D_{\Phi_2})$ is contained in a Euclidean annulus centered at $y$ with constant ratio. In particular, the image of the Euclidean circle centered at $x$ contained in $A_{\lambda_1s,s}^\bullet(x; D_{\Phi_1})$ is contained in a Euclidean annulus centered at $y$ with constant ratio. Combining this with condition~\eqref{it:def-quasiconformal-3}, we conclude that $f$ must be quasiconformal.}
    \label{fig:quasiconformality}
\end{figure}

Before proceeding, we record the following observations about general quasisymmetric mappings.

\begin{remark}\label{rem:quasisymmetric}
    Let $f \colon (X, D_X) \to (Y, D_Y)$ be an $\eta$-quasisymmetric homeomorphism. 
    \begin{enumerate}
        \item For each $a > 1$, $x \in X$, and $s > 0$, there exists $t > 0$ such that 
        \begin{equation*}
            f(A_{s,as}(x; D_X)) \subset A_{t,\eta(a)t}(f(x); D_Y). 
        \end{equation*}
        Similarly, for each $b \in (0, 1)$, $x \in X$, and $s > 0$, there exists $t > 0$ such that 
        \begin{equation*}
            A_{bt,t}(f(x); D_Y) \subset f(A_{\eta^{-1}(b)s,s}(x; D_X)). 
        \end{equation*}
        \item\label{it:quasisymmetric-1} For each $\mu, \nu \in (0, 1)$, there exists $C = C(\eta, \mu, \nu) \ge 1$ such that the following is true: Let $x \in X$ and $s, t > 0$. Suppose that $\iota \colon \BN \to \BN$ is an assignment such that
        \begin{equation*}
            f(A_{\mu^{j + 1}s,\mu^js}(x; D_X)) \cap A_{\nu^{\iota(j) + 1}t,\nu^{\iota(j)}t}(f(x); D_Y) \neq \emptyset, \quad \forall j \in \BN. 
        \end{equation*}
        Then:
        \begin{itemize}
            \item $\iota([n, 2n]_\BZ) \subset [C^{-1}n, Cn]_\BZ$ for all sufficiently large $n \in \BN$.
            \item $\#\iota^{-1}(n) \le C$ for all $n \in \BN$.
        \end{itemize}
        \item\label{it:quasisymmetric-2} The following elementary observation is immediate to verify: For each $C \ge 1$ and $p \in (0, 1)$, there exists $q = q(C, p) \in (0, 1)$ such that the following is true: Let $\iota \colon \BZ \to \BZ$ and $S, T \subset \BZ$. Suppose that the following conditions are satisfied:
        \begin{itemize}
            \item $\iota([n, 2n]_\BZ) \subset [C^{-1}n, Cn]_\BZ$ and $\#\iota^{-1}(n) \le C$ for all sufficiently large $n \in \BN$.
            \item $\#(S \cap [n, 2n]_\BZ) \ge pn$ and $\#(T \cap [n, 2n]_\BZ) \ge qn$ for all sufficiently large $n \in \BN$. 
        \end{itemize}
        Then for every sufficiently large $n \in \BN$, there exists $j \in S \cap [n, 2n]_\BZ$ such that $[\iota(j) - C, \iota(j) + C]_\BZ \subset T$. 
    \end{enumerate}
\end{remark}

\begin{proof}[Proof of \Cref{lem:quasiconformality}]
    See \Cref{fig:quasiconformality} for an illustration.  
    
    By the scaling property, it suffices to show that, almost surely, every $\eta$-quasisymmetric embedding of $(B_1^\bullet(0; D_{\Phi_1}), D_{\Phi_1})$ into $(B_1^\bullet(0; D_{\Phi_2}), D_{\Phi_2})$ is $K$-quasiconformal for some $K = K(\eta) \ge 1$. Fix $T > 0$ and $\varepsilon_\ast \in (0, 1)$. Let $E_1(T,\varepsilon_\ast)$ (resp.~$E_2(T,\varepsilon_\ast)$) be defined in the same manner as the event $E(T,\varepsilon_\ast)$ in \Cref{lem:quasiconformality-proof}, but with $\Phi_1$ (resp.~$\Phi_2$) in place of $\Phi$. By \Cref{lem:net}, there almost surely exist $T > 0$ and $\varepsilon_\ast \in (0,1)$ such that both $E_1(T,\varepsilon_\ast)$ and $E_2(T,\varepsilon_\ast)$ occur. Thus, it suffices to work on these events.

    By \Cref{lem:quasiconformality-proof},~\eqref{it:quasiconformality-proof-0} and the Borel--Cantelli lemma, there exists $\lambda_1 \in (0, 1)$ such that, almost surely, for each sufficiently small $\varepsilon \in \{\lambda_1^n\}_{n \in \BZ}$ and each $x \in B_1^\bullet(0; D_{\Phi_1})$, there are at least $\frac12\log_{1/\lambda_1}(1/\varepsilon)$ values of $s \in [\varepsilon^2, \varepsilon] \cap \{\lambda_1^n\}_{n \in \BZ}$ for which $\Mod(A_{\lambda_1s,s}^\bullet(x; D_{\Phi_1})) > \frac{5\log(2)}{2\pi}$, in which case, by \Cref{lem:Teichmueller}, there exists $r > 0$ such that $\partial B_r(x) \subset A_{\lambda_1s,s}^\bullet(x; D_{\Phi_1})$. 

    Set $\lambda_2 \defeq 1/\eta(1/\lambda_1)$. Fix $b \in (0, 1)$ to be chosen later. By \Cref{lem:quasiconformality-proof},~\eqref{it:quasiconformality-proof-1} and the Borel--Cantelli lemma, there exists $M = M(\lambda_2, b) \ge 1$ such that, almost surely, for each sufficiently small $\varepsilon \in \{\lambda_2^n\}_{n \in \BZ}$ and each $y \in B_1^\bullet(0; D_{\Phi_2})$, there are at least $b\log_{1/\lambda_2}(1/\varepsilon)$ values of $t \in [\varepsilon^2, \varepsilon] \cap \{\lambda_2^n\}_{n \in \BZ}$ for which $M^{-1} \le \Mod(A_{\lambda_2t,t}^\bullet(y; D_{\Phi_2})) \le M$. 

    Let $f$ be an $\eta$-quasisymmetric embedding of $(B_1^\bullet(0; D_{\Phi_1}), D_{\Phi_1})$ into $(B_1^\bullet(0; D_{\Phi_2}), D_{\Phi_2})$. We now prove that $f$ is quasiconformal by verifying~\eqref{it:def-quasiconformal-3}. Fix $x \in B_1^\bullet(0; D_{\Phi_1})$ and $y \defeq f(x)$. Consider the assignment
    \begin{equation*}
        \iota \colon \{\lambda_1^j\}_{j \in \BZ} \to \{\lambda_2^j\}_{j \in \BZ} \colon s \mapsto \sup\!\left\{t \in \{\lambda_2^j\}_{j \in \BZ} : f(A_{\lambda_1s,s}^\bullet(x; D_{\Phi_1})) \cap A_{\lambda_2t,t}^\bullet(y; D_{\Phi_2}) \neq \emptyset\right\}. 
    \end{equation*}
    (Here, we observe that $f(A_{\lambda_1s,s}^\bullet(x; D_{\Phi_1})) \cap A_{\lambda_2t,t}^\bullet(y; D_{\Phi_2}) \neq \emptyset$ if and only if $f(A_{\lambda_1s,s}(x; D_{\Phi_1})) \cap A_{\lambda_2t,t}(y; D_{\Phi_2}) \neq \emptyset$.) By the choice of $\lambda_2$, we have $f(A_{\lambda_1s,s}^\bullet(x; D_{\Phi_1})) \subset A_{\lambda_2^2\iota(s),\iota(s)}^\bullet(y; D_{\Phi_2})$ (or, equivalently, $f(A_{\lambda_1s,s}(x; D_{\Phi_1})) \subset A_{\lambda_2^2\iota(s),\iota(s)}(y; D_{\Phi_2})$) for all $s \in \{\lambda_1^j\}_{j \in \BZ}$. 
    By \Cref{rem:quasisymmetric},~\eqref{it:quasisymmetric-1} and~\eqref{it:quasisymmetric-2}, we may choose $b = b(\eta, \lambda_1, \lambda_2)$ to be sufficiently close to $1$ so that for each sufficiently small $\varepsilon \in \{\lambda_1^n\}_{n \in \BZ}$, there exists $s \in [\varepsilon^2, \varepsilon] \cap \{\lambda_1^n\}_{n \in \BZ}$ such that $\partial B_r(x) \subset A_{\lambda_1s,s}^\bullet(x; D_{\Phi_1})$ for some $r > 0$ and
    \begin{equation*}
        M^{-1} \le \Mod(A_{\lambda_2t,t}^\bullet(y; D_{\Phi_2})) \le M, \quad \forall t \in \{\lambda_2^2\iota(s), \lambda_2\iota(s), \iota(s), \lambda_2^{-1}\iota(s)\}. 
    \end{equation*}
    By \Cref{lem:Teichmueller}, we have
    \begin{gather*}
        \frac{\inf\{R > 0 : B_{\iota(s)}^\bullet(y; D_{\Phi_2}) \subset B_R(y)\}}{\sup\{R > 0 : B_R(y) \subset B_{\lambda_2^{-1}\iota(s)}^\bullet(y; D_{\Phi_2})\}} \le \frac1{16}\exp\!\left(\frac\pi{2\Mod(A_{\iota(s),\lambda_2^{-1}\iota(s)}^\bullet(y; D_{\Phi_2}))}\right) \le \re^{\pi M/2}/16; \\
        \frac{\inf\{R > 0 : B_{\lambda_2^3\iota(s)}^\bullet(y; D_{\Phi_2}) \subset B_R(y)\}}{\sup\{R > 0 : B_R(y) \subset B_{\lambda_2^2\iota(s)}^\bullet(y; D_{\Phi_2})\}} \le \frac1{16}\exp\!\left(\frac\pi{2\Mod(A_{\lambda_2^3\iota(s),\lambda_2^2\iota(s)}^\bullet(y; D_{\Phi_2}))}\right) \le \re^{\pi M/2}/16; \\
        \frac{\sup\{R > 0 : B_R(y) \subset B_{\lambda_2^{-1}\iota(s)}^\bullet(y; D_{\Phi_2})\}}{\inf\{R > 0 : B_{\lambda_2^3\iota(s)}^\bullet(y; D_{\Phi_2}) \subset B_R(y)\}} \le \exp\!\left(2\pi\Mod(A_{\lambda_2^3\iota(s),\lambda_2^{-1}\iota(s)}^\bullet(y; D_{\Phi_2}))\right) \le \re^{8\pi M}. 
    \end{gather*}
    This implies that
    \begin{align*}
        \frac{\inf\{R > 0 : B_{\iota(s)}^\bullet(y; D_{\Phi_2}) \subset B_R(y)\}}{\sup\{R > 0 : B_R(y) \subset B_{\lambda_2^2\iota(s)}^\bullet(y; D_{\Phi_2})\}} &= \frac{\inf\{R > 0 : B_{\iota(s)}^\bullet(y; D_{\Phi_2}) \subset B_R(y)\}}{\sup\{R > 0 : B_R(y) \subset B_{\lambda_2^{-1}\iota(s)}^\bullet(y; D_{\Phi_2})\}} \\
        &\times \frac{\sup\{R > 0 : B_R(y) \subset B_{\lambda_2^{-1}\iota(s)}^\bullet(y; D_{\Phi_2})\}}{\inf\{R > 0 : B_{\lambda_2^3\iota(s)}^\bullet(y; D_{\Phi_2}) \subset B_R(y)\}} \\
        &\times \frac{\inf\{R > 0 : B_{\lambda_2^3\iota(s)}^\bullet(y; D_{\Phi_2}) \subset B_R(y)\}}{\sup\{R > 0 : B_R(y) \subset B_{\lambda_2^2\iota(s)}^\bullet(y; D_{\Phi_2})\}} \\
        &\le \re^{\pi M/2}/16 \times \re^{8\pi M} \times \re^{\pi M/2}/16. 
    \end{align*}
    Set $H \defeq \re^{\pi M/2}/16 \times \re^{8\pi M} \times \re^{\pi M/2}/16$. Thus, we conclude that
    \begin{equation*}
        f(\partial B_r(x)) \subset f(A_{\lambda_1s,s}^\bullet(x; D_{\Phi_1})) \subset A_{\lambda_2^2\iota(s),\iota(s)}^\bullet(y; D_{\Phi_2}) \subset A_{r^\prime,Hr^\prime}(y)
    \end{equation*}
    for some $r^\prime > 0$. Since $r \to 0$ as $\varepsilon \to 0$, we conclude that
    \begin{equation*}
        \liminf_{r \to 0} \frac{\sup\{\lVert f(x) - f(z)\rVert : z \in \partial B_r(x)\}}{\inf\{\lVert f(x) - f(z)\rVert : z \in \partial B_r(x)\}} \le H, \quad \forall x \in B_1^\bullet(0; D_{\Phi_1}). 
    \end{equation*}
    We note that the constant $H$ depends only on $M$, which depends only on $\lambda_2$ and $b$. Moreover, $b$ depends only on $\eta$, $\lambda_1$, and $\lambda_2$, while $\lambda_1$ is a universal constant and $\lambda_2$ depends only on $\eta$ and $\lambda_1$. This completes the proof. 
\end{proof}

\subsection{Variants of \Cref{lem:quasiconformality}}

In the present subsection, we prove a whole-plane GFF variant of \Cref{lem:quasiconformality} (cf.~\Cref{lem:quasiconformality-GFF-GFF}). This variant will be useful in \Cref{section:main-proof}.

\begin{corollary}\label{lem:quasiconformality-cone-cone}
    For each distortion function $\eta$, there exists $K = K(\eta) \ge 1$ such that the following is true: Let $(\BC, \Phi_1; 0, \infty)$ and $(\BC, \Phi_2; 0, \infty)$ be a coupled pair of $\sqrt{8/3}$-LQG cones. Then, almost surely, for every open subset $U \subset \BC$, any $\eta$-quasisymmetric embedding of $(U, D_{\Phi_1})$ into $(\BC, D_{\Phi_2})$ is $K$-quasiconformal.
\end{corollary}

\begin{proof}
    Let $\eta^\prime \colon \BR \to \BC$ be an independent whole-plane space-filling SLE$_{\kappa^\prime}$ curve from $\infty$ to $\infty$, parameterized by the LQG area measure $\SM_{\Phi_1}$. Recall that for each deterministic $t \in \BR$, the $\sqrt{8/3}$-LQG surface $(\BC, \Phi_1; \eta^\prime(t), \infty)$ has the law of a $\sqrt{8/3}$-LQG cone. Thus, by \Cref{lem:quasiconformality} and the scaling property, almost surely, for each $t \in \BQ$ and $s \in \BQ_{>0}$, every $\eta$-quasisymmetric embedding of $(B_s^\bullet(\eta^\prime(t); D_{\Phi_1}), D_{\Phi_1})$ into $(\BC, D_{\Phi_2})$ is $K$-quasiconformal. Since for each $z \in U$, there exists $t \in \BQ$ and $s \in \BQ_{>0}$ such that $z \in B_s^\bullet(\eta^\prime(t); D_{\Phi_1}) \subset U$, this implies that any $\eta$-quasisymmetric embedding of $(U, D_{\Phi_1})$ into $(\BC, D_{\Phi_2})$ is $K$-quasiconformal. (Here, we recall the fact that $K$-quasiconformality is a local property.) This completes the proof. 
\end{proof}

\begin{corollary}\label{lem:quasiconformality-cone-GFF}
    For each distortion function $\eta$, there exists $K = K(\eta) \ge 1$ such that the following holds. Let $\Phi_1$ be a whole-plane GFF with fixed additive constant. Let $(\BC, \Phi_2; 0, \infty)$ be a $\sqrt{8/3}$-LQG cone, coupled with $\Phi_1$. Then, almost surely, for every open subset $U \subset \BC$, every $\eta$-quasisymmetric embedding of $(U, D_{\Phi_1})$ into $(\BC, D_{\Phi_2})$ is $K$-quasiconformal. The same statement holds with the roles of $\Phi_1$ and $\Phi_2$ interchanged.
\end{corollary}

\begin{proof}
    First, we consider the first portion of the assertion. Since $K$-quasiconformality is a local property, it suffices to consider the case where $U = B_r(z)$ for some $z \in \BQ^2$ and $r \in \BQ_{>0}$. Note that, by the Weyl scaling axiom, the additive constant of $\Phi_1$ affects neither $\eta$-quasisymmetry nor $K$-quasiconformality. Thus, by the scale and translation invariance of the whole-plane GFF, it suffices to consider the case where $U = B_{1/4}(1/2)$. 
    
    We may now assume without loss of generality that $\Phi_1$ is normalized so that its circle average over $\partial\BD$ is zero. Let $\widetilde\Phi_1$ be the circle average embedding of a $\sqrt{8/3}$-LQG cone. Recall that $\widetilde\Phi_1|_\BD$ agrees in law with $\Phi_1|_\BD - \sqrt{8/3}\log(\lVert\bullet\rVert)$. Suppose that $\Phi_1$, $\widetilde\Phi_1$, and $\Phi_2$ are coupled so that $\widetilde\Phi_1|_\BD = \Phi_1|_\BD - \sqrt{8/3}\log(\lVert\bullet\rVert)$ almost surely. Since $\lVert\widetilde\Phi_1 - \Phi_1\rVert$ is bounded above by a universal constant on $B_{1/4}(1/2)$, it follows that the internal metrics $D_{\widetilde\Phi_1}(\bullet,\bullet; B_{1/4}(1/2))$ and $D_{\Phi_1}(\bullet,\bullet; B_{1/4}(1/2))$ are almost surely bi-Lipschitz equivalent with universal constants.
    
    Let $f$ be an $\eta$-quasisymmetric embedding of $(B_{1/4}(1/2), D_{\Phi_1})$ into $(\BC, D_{\Phi_2})$. For each $z \in B_{1/4}(1/2)$, there exists a sufficiently small $r > 0$ such that the restrictions $D_{\widetilde\Phi_1}(\bullet, \bullet; B_{1/4}(1/2))|_{B_r(z) \times B_r(z)}$ and $D_{\Phi_1}(\bullet, \bullet; B_{1/4}(1/2))|_{B_r(z) \times B_r(z)}$ agree with $D_{\widetilde\Phi_1}(\bullet, \bullet)|_{B_r(z) \times B_r(z)}$ and $D_{\Phi_1}(\bullet, \bullet)|_{B_r(z) \times B_r(z)}$, respectively. In this case, $D_{\widetilde\Phi_1}(\bullet, \bullet)|_{B_r(z) \times B_r(z)}$ and $D_{\Phi_1}(\bullet, \bullet)|_{B_r(z) \times B_r(z)}$ are bi-Lipschitz equivalent with universal constants, which implies that $f|_{B_r(z)}$ is $\widetilde\eta$-quasisymmetric with respect to $D_{\widetilde\Phi_1}$ and $D_{\Phi_2}$ for some distortion function $\widetilde\eta = \widetilde\eta(\eta)$. Thus, by \Cref{lem:quasiconformality-cone-cone}, there exists $\widetilde K = \widetilde K(\widetilde\eta) \ge 1$ such that $f|_{B_r(z)}$ is $\widetilde K$-quasiconformal. Since $z \in B_{1/4}(1/2)$ was arbitrary, we conclude that $f$ is $\widetilde K$-quasiconformal. This completes the proof of the first portion of the assertion. 

    The second portion of the assertion follows immediately from the first portion, together with the facts that the inverse of an $\eta$-quasisymmetric homeomorphism is $1/\eta^{-1}(\bullet^{-1})$-quasisymmetric, and that the inverse of a $K$-quasiconformal homeomorphism is also $K$-quasiconformal.
\end{proof}

\begin{corollary}\label{lem:quasiconformality-GFF-GFF}
    For each distortion function $\eta$, there exists $K = K(\eta) \ge 1$ such that the following is true: Let $\Phi_1$ and $\Phi_2$ be a coupled pair of whole-plane GFFs with fixed additive constants. Then, almost surely, for every open subset $U \subset \BC$, every $\eta$-quasisymmetric embedding of $(U, D_{\Phi_1})$ into $(\BC, D_{\Phi_2})$ is $K$-quasiconformal.
\end{corollary}

\begin{proof}
    This follows immediately from the second portion of \Cref{lem:quasiconformality-cone-GFF}, together with a similar argument to the argument used in the proof of the first portion of \Cref{lem:quasiconformality-cone-GFF}.
\end{proof}

\section{Proof of the main results}\label{section:main-proof}

In the present section, we complete the proof of the main results, \Cref{thm:main-ind,thm:main-aut}. The main input is \Cref{lem:nonexistence}, which in turn follows from \Cref{lem:matrix}.

\begin{proposition}\label{lem:nonexistence}
    Let $\Phi_1$ and $\Phi_2$ be independent whole-plane GFFs with fixed additive constants. Then, almost surely, there is no quasisymmetric embedding of $(\BD, D_{\Phi_1})$ into $(\BC, D_{\Phi_2})$.
\end{proposition}

\begin{lemma}\label{lem:matrix}
    In the notation of \Cref{lem:nonexistence}, let $\xx$ be sampled uniformly from the Lebesgue measure on $\BD$, independently of everything else. Fix a distortion function $\eta$. Then, for each invertible matrix $A \in \mathop{\mathrm{GL}}(2,\BR)$, there exists $\delta = \delta(\eta,A) > 0$ such that, almost surely, there is no $\eta$-quasisymmetric embedding $f$ from $(\BD, D_{\Phi_1})$ into $(\BC, D_{\Phi_2})$ for which $Df(\xx)$ exists, is invertible, and satisfies $\lVert Df(\xx) - A\rVert \le \delta$, where $\lVert\bullet\rVert$ denotes the spectral norm.
\end{lemma}

Before proceeding, we recall from \cite[Theorem~11.3]{MR1800917} that every quasisymmetric homeomorphism $f \colon (X, D_X) \to (Y, D_Y)$ between connected metric spaces is $\eta$-quasisymmetric for some distortion function $\eta$ of the form $\eta(\bullet) = C(\bullet^\alpha \vee \bullet^{1/\alpha})$, where $C > 0$ and $\alpha \ge 1$.

\begin{proof}[Proof of \Cref{lem:nonexistence} assuming \Cref{lem:matrix}]
    By \cite[Theorem~11.3]{MR1800917}, it suffices to consider the case of a fixed distortion function $\eta$. By \Cref{lem:Borel}, for each $R > 0$, the event that there exists an $\eta$-quasisymmetric embedding of $(\BD, D_{\Phi_1})$ into $(B_R(0), D_{\Phi_2})$ is measurable. Thus, the event that there exists an $\eta$-quasisymmetric embedding of $(\BD, D_{\Phi_1})$ into $(\BC, D_{\Phi_2})$ is measurable.

    Suppose, by way of contradiction, that with positive probability there exists an $\eta$-quasisymmetric embedding of $(\BD, D_{\Phi_1})$ into $(\BC, D_{\Phi_2})$. By \Cref{lem:quasiconformality-GFF-GFF} and the fact that the derivative of a quasiconformal mapping exists and is invertible at Lebesgue almost every point, we conclude that, with positive probability, there exists an $\eta$-quasisymmetric embedding $f$ from $(\BD, D_{\Phi_1})$ into $(\BC, D_{\Phi_2})$ such that $Df(\xx)$ exists and is invertible. On the other hand, since $\mathop{\mathrm{GL}}(2,\BR)$ is second countable, there exists a countable subcover of
    \begin{equation*}
        \bigcup_{A \in \mathop{\mathrm{GL}}(2,\BR)} \{J \in \mathop{\mathrm{GL}}(2,\BR) : \lVert J - A\rVert \le \delta(\eta, A)\}.
    \end{equation*}
    By \Cref{lem:matrix}, this implies that, almost surely, there is no $\eta$-quasisymmetric embedding $f$ from $(\BD, D_{\Phi_1})$ into $(\BC, D_{\Phi_2})$ such that $Df(\xx)$ exists and is invertible, a contradiction. This completes the proof.
\end{proof}

\Cref{subsection:pattern,subsection:proof-matrix} are devoted to the proof of \Cref{lem:matrix}. The intuition for the proof is as follows. Since $f^{-1}$ is differentiable at $f(\xx)$ with Jacobian matrix close to $A^{-1}$, it maps each sufficiently small Euclidean annulus around $f(\xx)$ to a subset close to an elliptical annulus around $\xx$. Since $f$ is $\eta$-quasisymmetric, this implies that the ``pattern'' of each such Euclidean annulus must match the ``pattern'' of the corresponding elliptical annulus. However, we will carefully design the ``patterns'' so that, given any prescribed ``pattern'', the ``pattern'' of the Euclidean annulus is different from the prescribed one with high probability. Thus, using independence across scales for the whole-plane GFF (cf.~\Cref{lem:independence-across-scales}), we will show that, almost surely given $\Phi_1$ and $\xx$, with probability one there is no point around which the ``patterns'' of the Euclidean annuli at all scales match the ``patterns'' of the elliptical annuli around $\xx$. This implies that $f$ cannot exist.

\subsection{Constructing obstructions}\label{subsection:pattern}

In the present subsection, we construct a suitable ``pattern'' associated with a given Euclidean annulus, and show that, for any prescribed ``pattern'', with high probability the ``pattern'' of the Euclidean annulus is not the prescribed one. We will achieve this as follows. First, we design ``toy patterns'' on Euclidean balls, so that, for any given ``toy pattern'', with uniformly positive probability the ``toy pattern'' associated with a Euclidean ball is not the given one; see \Cref{lem:E-properties}. Then, by placing many Euclidean balls inside a Euclidean annulus and defining the ``pattern'' of the annulus to be the data consisting of the ``toy patterns'' of all these balls, we will be able to ``bootstrap'' the probability to be high; see \Cref{lem:F-properties}.

Let $\Phi_1$ and $\Phi_2$ be as in \Cref{lem:nonexistence}. Fix a distortion function $\eta$ and an invertible $2 \times 2$ matrix $A \in \mathop{\mathrm{GL}}(2, \BR)$.

We shall refer to an interval as \emph{($\eta$-)admissible} if it is of the form $[1/\eta(1/t), \eta(t)]$ for some $t > 0$. Note that if $f \colon (X, D_X) \to (Y, D_Y)$ is an $\eta$-quasisymmetric homeomorphism and $x,y,z \in X$ satisfy $D_X(x,y)/D_X(x,z) = t$, then the ratio $D_Y(f(x),f(y))/D_Y(f(x),f(z))$ must belong to the admissible interval $[1/\eta(1/t), \eta(t)]$. 

\begin{figure}[ht!]
    \centering
    \includegraphics[width=.6\linewidth]{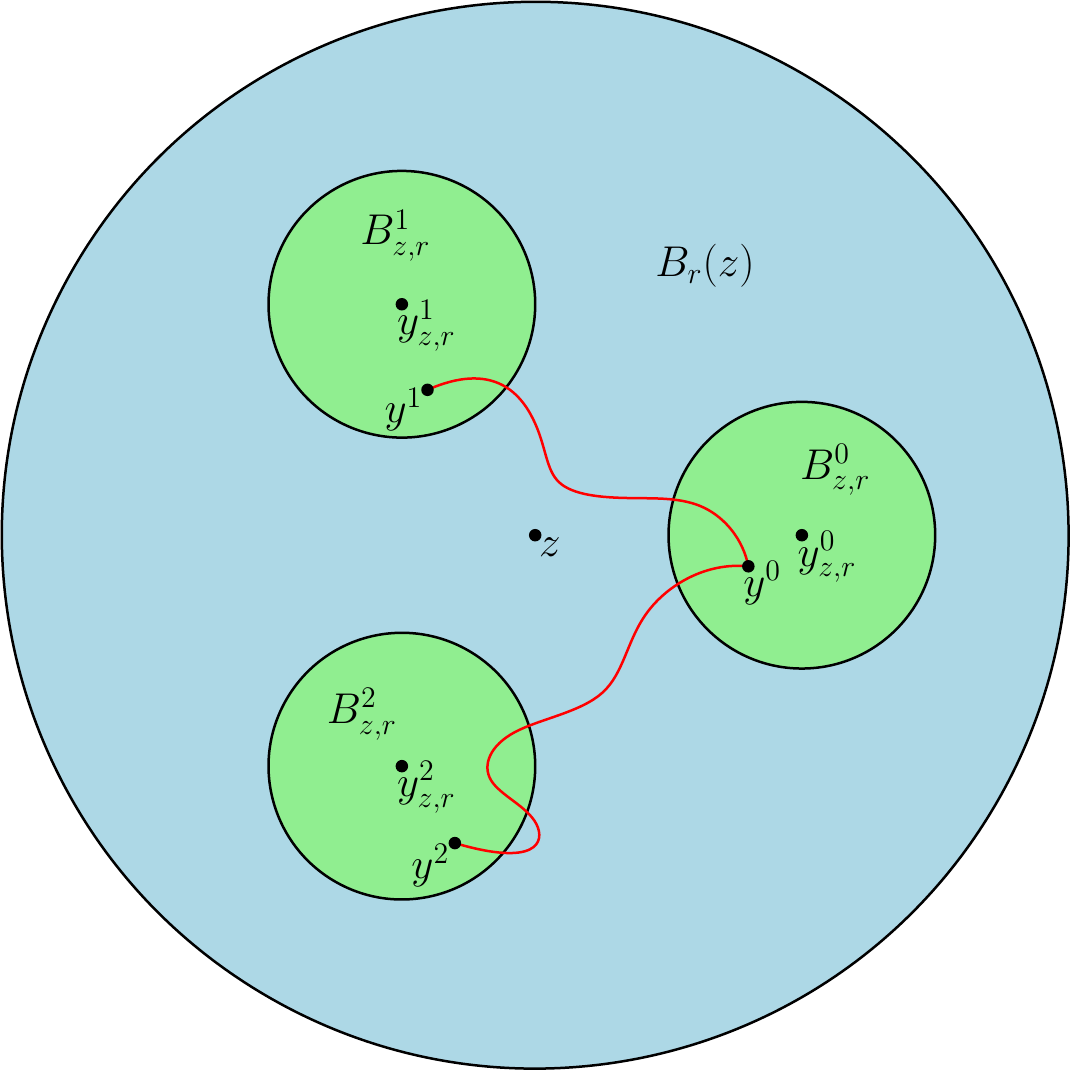}
    \caption{Illustration of the objects involved in the definition of the event $E_{z,r,I}$. The light blue region depicts the Euclidean ball $B_r(z)$. The light green Euclidean balls $B_{z,r}^0$, $B_{z,r}^1$, and $B_{z,r}^2$ have radius $r/4$ and are centered at $y_{z,r}^0 = z + r/2$, $y_{z,r}^1 = z + r\re^{2\pi\ri/3}/2$, and $y_{z,r}^2 = z + r\re^{4\pi\ri/3}/2$, respectively. Condition~\eqref{it:def-E-0} in the definition of $E_{z,r,I}$ requires that, for each $y^0 \in B_{z,r}^0$, $y^1 \in B_{z,r}^1$, and $y^2 \in B_{z,r}^2$, the ratio $D_{\Phi_2}(y^0, y^1)/D_{\Phi_2}(y^0, y^2)$ is not contained in the prescribed admissible interval $I$. Moreover, condition~\eqref{it:def-E-1} ensures that the $D_{\Phi_2}$-geodesics connecting these points are contained in $B_r(z)$. By the definition of admissible intervals, the event $E_{z,r,I}$ prevents certain triples of points from being mapped $\eta$-quasisymmetrically into $B_{z,r}^0$, $B_{z,r}^1$, and $B_{z,r}^2$, respectively.}
    \label{fig:E}
\end{figure}

Let $z \in \BC$ and $r > 0$. We shall write
\begin{gather*}
    y_{z,r}^0 \defeq z + r/2; \quad y_{z,r}^1 \defeq z + r\re^{2\pi\ri/3}/2; \quad y_{z,r}^2 \defeq z + r\re^{4\pi\ri/3}/2; \\
    B_{z,r}^0 \defeq B_{r/4}(y_{z,r}^0); \quad B_{z,r}^1 \defeq B_{r/4}(y_{z,r}^1); \quad B_{z,r}^2 \defeq B_{r/4}(y_{z,r}^2). 
\end{gather*}
Let $I$ be an admissible interval. We shall write $E_{z,r,I}$ for the following event. 
\begin{enumerate}[label=(\alph*),ref=\alph*]
    \item\label{it:def-E-0} We have
    \begin{equation*}
        \frac{D_{\Phi_2}(y^0, y^1)}{D_{\Phi_2}(y^0, y^2)} \notin I, \quad \forall y^0 \in B_{z,r}^0, \ \forall y^1 \in B_{z,r}^1, \ \forall y^2 \in B_{z,r}^2. 
    \end{equation*}
    \item\label{it:def-E-1} We have $D_{\Phi_2}(y, y^\prime) \le D_{\Phi_2}(y, \partial B_r(z))$ for all $y, y^\prime \in B_{3r/4}(z)$. 
\end{enumerate}
See \Cref{fig:E} for an illustration. 

\begin{lemma}\label{lem:E-properties}
    \begin{enumerate}
        \item\label{it:E-properties-0} The event $E_{z,r,I}$ is almost surely determined by $\Phi_2|_{B_r(z)}$ modulo additive constant.
        \item\label{it:E-properties-1} The probability $\BP\lbrack E_{z,r,I}\rbrack$ does not depend on $z$ and $r$. 
        \item\label{it:E-properties-2} There exists $p = p(\eta) \in (0, 1)$ such that $\BP\lbrack E_{z,r,I}\rbrack \ge p$ for all admissible intervals $I$.
    \end{enumerate}
\end{lemma}

\begin{proof}
    By the Weyl scaling axiom, the additive constant of $\Phi_2$ does not affect the occurrence of $E_{z,r,I}$. On the event that condition~\eqref{it:def-E-1} holds, any path between two points of $B_{3r/4}(z)$ that exits $B_r(z)$ cannot be a $D_{\Phi_2}$-geodesic. This completes the proof of assertion~\eqref{it:E-properties-0}. Assertion~\eqref{it:E-properties-1} follows immediately from the scale and translation invariance of the whole-plane GFF. 

    Finally, we consider assertion~\eqref{it:E-properties-2}. We \emph{claim} that for each $M > 0$, it holds with positive probability that 
    \begin{equation}\label{eq:E-properties-proof-0}
        \frac{D_{\Phi_2}(y^0, y^1)}{D_{\Phi_2}(y^0, y^2)} > M, \quad \forall y^0 \in B_{z,r}^0, \ \forall y^1 \in B_{z,r}^1, \ \forall y^2 \in B_{z,r}^2
    \end{equation}
    and condition~\eqref{it:def-E-1} holds. 
    
    We may now assume without loss of generality that $B_r(z) \subset \BD$ and that $\Phi_2$ is normalized so that its circle average over $\partial\BD$ is zero. We may choose a sufficiently large constant $C > 0$ such that it holds with positive probability that 
    \begin{gather*}
        D_{\Phi_2}(y^0, y^2; B_{7r/8}(z) \setminus B_{3r/8}(y_{z,r}^1)) < C, \quad \forall y^0 \in B_{z,r}^0, \ \forall y^2 \in B_{z,r}^2; \\
        D_{\Phi_2}(\partial B_{9r/32}(y_{z,r}^1), \partial B_{5r/16}(y_{z,r}^1)) > C^{-1}; \\
        D_{\Phi_2}(y, y^\prime; B_{7r/8}(z)) < C, \quad \forall y, y^\prime \in B_{3r/4}(z); \\
        D_{\Phi_2}(\partial B_{29r/32}(z), \partial B_{15r/16}(z)) > C^{-1}.
    \end{gather*}
    Choose a deterministic bump function $\phi$ (resp.~$\phi^\prime$) supported on $A_{r/4,3r/8}(y_{z,r}^1)$ (resp.~$A_{7r/8,r}(z)$) and equal to the constant $\sqrt6\log(C^2M)$ (resp.~$\sqrt6\log(C^4M)$) on $A_{9r/32,5r/16}(y_{z,r}^1)$ (resp.~$A_{29r/32,15r/16}(z)$). Thus, by the Weyl scaling axiom, it holds with positive probability that 
    \begin{gather*}
        \begin{multlined}
            D_{\Phi_2 + \phi + \phi^\prime}(y^0, y^2; B_{7r/8}(z) \setminus B_{3r/8}(y_{z,r}^1)) = D_{\Phi_2}(y^0, y^2; B_{7r/8}(z) \setminus B_{3r/8}(y_{z,r}^1)) < C, \\
            \forall y^0 \in B_{z,r}^0, \ \forall y^2 \in B_{z,r}^2;
        \end{multlined} \\
        D_{\Phi_2 + \phi + \phi^\prime}(\partial B_{9r/32}(y_{z,r}^1), \partial B_{5r/16}(y_{z,r}^1)) = C^2M D_{\Phi_2}(\partial B_{9r/32}(y_{z,r}^1), \partial B_{5r/16}(y_{z,r}^1)) > CM; \\
        D_{\Phi_2 + \phi + \phi^\prime}(y, y^\prime; B_{7r/8}(z)) \le C^2M D_{\Phi_2}(y, y^\prime; B_{7r/8}(z)) < C^3M, \quad \forall y, y^\prime \in B_{3r/4}(z); \\
        D_{\Phi_2 + \phi + \phi^\prime}(\partial B_{29r/32}(z), \partial B_{15r/16}(z)) = C^4M D_{\Phi_2}(\partial B_{29r/32}(z), \partial B_{15r/16}(z)) > C^3M.
    \end{gather*}
    Since the support of $\phi + \phi^\prime$ is contained in $\BD$, the laws of $\Phi_2$ and $\Phi_2 + \phi + \phi^\prime$ are mutually absolutely continuous. This implies that it holds with positive probability that 
    \begin{gather*}
        D_{\Phi_2}(y^0, y^2; B_{7r/8}(z) \setminus B_{3r/8}(y_{z,r}^1)) < C, \quad \forall y^0 \in B_{z,r}^0, \ \forall y^2 \in B_{z,r}^2; \\
        D_{\Phi_2}(\partial B_{9r/32}(y_{z,r}^1), \partial B_{5r/16}(y_{z,r}^1)) > CM; \\
        D_{\Phi_2}(y, y^\prime; B_{7r/8}(z)) < C^3M, \quad \forall y, y^\prime \in B_{3r/4}(z); \\
        D_{\Phi_2}(\partial B_{29r/32}(z), \partial B_{15r/16}(z)) > C^3M, 
    \end{gather*}
    in which case~\eqref{eq:E-properties-proof-0} and condition~\eqref{it:def-E-1} hold. This completes the proof of the \emph{claim}.  

    We may choose a sufficiently large constant $M_1 > 0$ such that it holds with positive probability that 
    \begin{equation}\label{eq:E-properties-proof-1}
        \frac{D_{\Phi_2}(y^0, y^1)}{D_{\Phi_2}(y^0, y^2)} < M_1, \quad \forall y^0 \in B_{z,r}^0, \ \forall y^1 \in B_{z,r}^1, \ \forall y^2 \in B_{z,r}^2
    \end{equation}
    and condition~\eqref{it:def-E-1} holds. Note that $[1/\eta(1/t), \eta(t)]$ shrinks to $0$ as $t \to 0$ and tends to $\infty$ as $t \to \infty$. Moreover, the assignment $t \mapsto [1/\eta(1/t), \eta(t)]$ is continuous. Thus, we may choose another sufficiently large constant $M_2 > M_1$ such that no admissible interval intersects both $[0, M_1]$ and $[M_2, \infty)$. Thus, each event $E_{z,r,I}$ contains either the event that~\eqref{eq:E-properties-proof-1} and condition~\eqref{it:def-E-1} hold, or the event that~\eqref{eq:E-properties-proof-0} (with $M_2$ in place of $M$) and condition~\eqref{it:def-E-1} hold. This completes the proof of assertion~\eqref{it:E-properties-2}. 
\end{proof}

Fix a sufficiently small parameter $a \in (0, 1)$ to be chosen later. 

Let $\zz \in \BC$ and $\rr > 0$. We define the collection of ``test points''
\begin{equation*}
    Z_{\zz,\rr} \defeq \{z_{\zz,\rr}^1, \dots, z_{\zz,\rr}^{\lfloor a^{-1}\rfloor}\}, \quad \text{where } z_{\zz,\rr}^j \defeq \zz + (3\rr/4)\re^{2\pi\ri aj}. 
\end{equation*}
Let $\SCI \defeq (I^1, \dots, I^{\lfloor a^{-1}\rfloor})$ be a sequence of admissible intervals. Then we shall write 
\begin{equation*}
    F_{\zz,\rr,\SCI} \defeq \bigcup_{j \in [1, a^{-1}]_\BZ} E_{z_{\zz,\rr}^j,a\rr,I^j}.
\end{equation*}
(Here, we note that $\lVert z_{\zz,\rr}^{j - 1} - z_{\zz,\rr}^j\rVert \approx 3\pi a\rr/2$ when $a$ is sufficiently small.) To lighten notation, we shall write
\begin{gather*}
    y_{\zz,\rr}^{j,0} \defeq y_{z_{\zz,\rr}^j,a\rr}^0; \quad y_{\zz,\rr}^{j,1} \defeq y_{z_{\zz,\rr}^j,a\rr}^1; \quad y_{\zz,\rr}^{j,2} \defeq y_{z_{\zz,\rr}^j,a\rr}^2; \\
    B_{\zz,\rr}^{j,0} \defeq B_{z_{\zz,\rr}^j,a\rr}^0; \quad B_{\zz,\rr}^{j,1} \defeq B_{z_{\zz,\rr}^j,a\rr}^1; \quad B_{\zz,\rr}^{j,2} \defeq B_{z_{\zz,\rr}^j,a\rr}^2. 
\end{gather*}
See \Cref{fig:F} for an illustration. 

\begin{figure}[ht!]
    \centering
    \includegraphics[width=.8\linewidth]{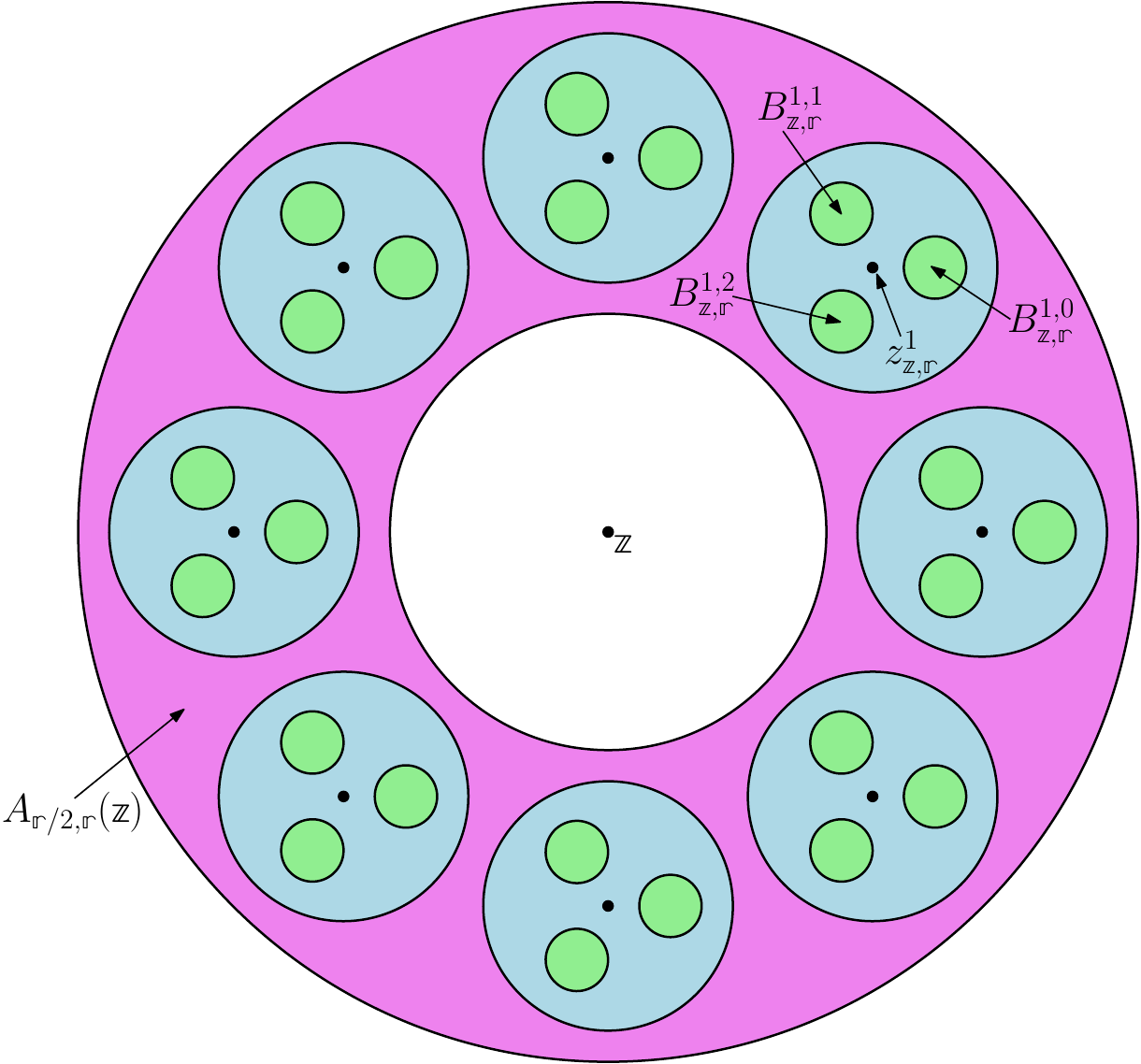}
    \caption{Illustration of the objects involved in the definition of the event $F_{\zz,\rr,\SCI}$ (with $a = 1/8$). The violet region depicts the Euclidean annulus $A_{\rr/2,\rr}(\zz)$. The black dots in this annulus are the points $z_{\zz,\rr}^1, \dots, z_{\zz,\rr}^{\lfloor a^{-1}\rfloor}$. The light green Euclidean balls are $B_{z_{\zz,\rr}^j,a\rr}^\bullet$, abbreviated as $B_{\zz,\rr}^{j,\bullet}$, for $j \in [1,a^{-1}]_\BZ$ and $\bullet \in \{0,1,2\}$. Their centers are $y_{z_{\zz,\rr}^j,a\rr}^\bullet$, abbreviated as $y_{\zz,\rr}^{j,\bullet}$. The event $F_{\zz,\rr,\SCI}$ requires that at least one of the events $E_{z_{\zz,\rr}^j,a\rr,I^j}$, for $j \in [1,a^{-1}]_\BZ$, occurs.}
    \label{fig:F}
\end{figure}

\begin{lemma}\label{lem:F-properties}
    \begin{enumerate}
        \item\label{it:F-properties-0} The event $F_{\zz,\rr,\SCI}$ is almost surely determined by $\Phi_2|_{A_{\rr/2,\rr}(\zz)}$ modulo additive constant. 
        \item\label{it:F-properties-1} The probability $\BP\lbrack F_{\zz,\rr,\SCI}\rbrack$ does not depend on $\zz$ and $\rr$. 
        \item\label{it:F-properties-2} For each $q \in (0,1)$, we may choose the parameter $a$ sufficiently small, depending only on $\eta$ and $q$, so that $\BP\lbrack F_{\zz,\rr,\SCI}\rbrack \ge q$ for every sequence of admissible intervals $\SCI$.
    \end{enumerate}
\end{lemma}

\begin{proof}
    By \Cref{lem:E-properties},~\eqref{it:E-properties-0}, the event $F_{\zz,\rr,\SCI}$ is almost surely determined by the restriction of $\Phi_2$ to $\bigcup_{j \in [1, a^{-1}]_\BZ} B_{a\rr}(z_{\zz,\rr}^j)$ modulo additive constant. Since $\bigcup_{j \in [1, a^{-1}]_\BZ} B_{a\rr}(z_{\zz,\rr}^j) \subset A_{\rr/2,\rr}(\zz)$, this completes the proof of assertion~\eqref{it:F-properties-0}. Assertion~\eqref{it:F-properties-1} follows immediately from the scale and translation invariance of the whole-plane GFF. Assertion~\eqref{it:F-properties-2} follows immediately from \Cref{lem:E-properties},~\eqref{it:E-properties-0} and~\eqref{it:E-properties-2} and \Cref{lem:independence-across-balls}. 
\end{proof}

We also record the following elementary lemma. It will be applied in \Cref{subsection:proof-matrix} with $\yy = f(\xx)$ and $J = Df(\xx)$.

\begin{lemma}\label{lem:linearization}
    There exists $\delta = \delta(A, a) > 0$ such that the following is true: Let $f^{-1} \colon B_\rr(\zz) \to \BC$ be an embedding. Suppose that there exists $J \in \mathop{\mathrm{GL}}(2, \BR)$ with $\lVert J - A\rVert \le \delta$ and $\lVert\yy - \zz\rVert \le \delta\rr$ such that
    \begin{equation}\label{eq:linearization}
        \lVert f^{-1}(y) - J^{-1}(y - \yy) - f^{-1}(\yy)\rVert \le \delta\lVert y - \yy\rVert, \quad \forall y \in B_\rr(\zz). 
    \end{equation}
    Then 
    \begin{equation*}
        A^{-1}(y_{\zz,\rr}^{j,\bullet} - \zz) + f^{-1}(\yy) \in f^{-1}(B_{\zz,\rr}^{j,\bullet}), \quad \forall j \in [1, a^{-1}]_\BZ, \ \forall \bullet \in \{0, 1, 2\}. 
    \end{equation*}
    (Here, we note that $y_{\zz,\rr}^{j,\bullet} - \zz = y_{0,\rr}^{j,\bullet}$ does not depend on $\zz$.)
\end{lemma}

\begin{proof}
    Indeed, 
    \begin{align*}
        \lVert f^{-1}(y_{\zz,\rr}^{j,\bullet}) - A^{-1}(y_{\zz,\rr}^{j,\bullet} - \zz) - f^{-1}(\yy)\rVert &\le \lVert f^{-1}(y_{\zz,\rr}^{j,\bullet}) - J^{-1}(y_{\zz,\rr}^{j,\bullet} - \yy) - f^{-1}(\yy)\rVert \\
        &+ \lVert(J^{-1} - A^{-1})(y_{\zz,\rr}^{j,\bullet} - \yy)\rVert + \lVert A^{-1}(\yy - \zz)\rVert \\
        &\le \delta\lVert y_{\zz,\rr}^{j,\bullet} - \yy\rVert + \lVert J^{-1} - A^{-1}\rVert \lVert y_{\zz,\rr}^{j,\bullet} - \yy\rVert + \lVert A^{-1}\rVert \lVert\yy - \zz\rVert \\
        &\le 2\delta\rr + 2\lVert J^{-1} - A^{-1}\rVert\rr + \delta\lVert A^{-1}\rVert\rr. 
    \end{align*}
    Since $\lVert J^{-1} - A^{-1}\rVert \to 0$ as $\delta \to 0$, it suffices to verify that $B_{a\lVert A\rVert^{-1}\rr/8}(f^{-1}(y_{\zz,\rr}^{j,\bullet})) \subset f^{-1}(B_{\zz,\rr}^{j,\bullet})$ whenever $\delta$ is sufficiently small (how small depends only on $A$ and $a$). (Here, we note that $B_{\zz,\rr}^{j,\bullet}$ has radius $a\rr/4$.) It suffices to verify that for each $y \in B_\rr(\zz)$ with $\lVert y - y_{\zz,\rr}^{j,\bullet}\rVert \ge a\rr/4$, we have $\lVert f^{-1}(y) - f^{-1}(y_{\zz,\rr}^{j,\bullet})\rVert \ge a\lVert A\rVert^{-1}\rr/8$. Indeed, 
    \begin{align*}
        \lVert f^{-1}(y) - f^{-1}(y_{\zz,\rr}^{j,\bullet})\rVert &\ge \lVert J^{-1}(y - \yy) - J^{-1}(y_{\zz,\rr}^{j,\bullet} - \yy)\rVert - \lVert f^{-1}(y) - J^{-1}(y - \yy) - f^{-1}(\yy)\rVert \\
        &- \lVert f^{-1}(y_{\zz,\rr}^{j,\bullet}) - J^{-1}(y_{\zz,\rr}^{j,\bullet} - \yy) - f^{-1}(\yy)\rVert \\
        &\ge \lVert J\rVert^{-1} \lVert y - y_{\zz,\rr}^{j,\bullet}\rVert - \delta\lVert y - \yy\rVert - \delta\lVert y_{\zz,\rr}^{j,\bullet} - \yy\rVert \\
        &\ge a(\lVert A\rVert + \delta)^{-1}\rr/4 - 4\delta\rr. 
    \end{align*}
    This completes the proof. 
\end{proof}

\subsection{Proof of \Cref{lem:matrix}}\label{subsection:proof-matrix}

Fix $A \in \mathop{\mathrm{GL}}(2, \BR)$. Fix $a \in (0, 1)$ to be chosen later. Let $\delta = \delta(A, a) > 0$ be as in \Cref{lem:linearization}. 

It suffices to show that, almost surely given $\Phi_1$ and $\xx$, the conditional probability is zero that there exists an $\eta$-quasisymmetric embedding $f$ from $(\BD, D_{\Phi_1})$ into $(\BC, D_{\Phi_2})$ for which $Df(\xx)$ exists, is invertible, and satisfies $\lVert Df(\xx) - A\rVert \le \delta$. Henceforth, we fix $\Phi_1$ and $\xx$.

For each $\rr > 0$ and $j \in [1, a^{-1}]_\BZ$, set $I_\rr^j$ to be the admissible interval
\begin{equation*}
    \left\lbrack1/\eta\!\left(\frac{D_{\Phi_1}(A^{-1}(y_{0,\rr}^{j,0}) + \xx, A^{-1}(y_{0,\rr}^{j,2}) + \xx)}{D_{\Phi_1}(A^{-1}(y_{0,\rr}^{j,0}) + \xx, A^{-1}(y_{0,\rr}^{j,1}) + \xx)}\right), \eta\!\left(\frac{D_{\Phi_1}(A^{-1}(y_{0,\rr}^{j,0}) + \xx, A^{-1}(y_{0,\rr}^{j,1}) + \xx)}{D_{\Phi_1}(A^{-1}(y_{0,\rr}^{j,0}) + \xx, A^{-1}(y_{0,\rr}^{j,2}) + \xx)}\right)\right\rbrack.
\end{equation*}
Note that $I_\rr^j$ is almost surely determined by $\Phi_1$ and $\xx$. 

Suppose by way of contradiction that there exists an $\eta$-quasisymmetric embedding $f$ from $(\BD, D_{\Phi_1})$ into $(\BC, D_{\Phi_2})$ for which $Df(\xx)$ exists, is invertible, and satisfies $\lVert Df(\xx) - A\rVert \le \delta$. Set $\yy \defeq f(\xx)$. Then $Df^{-1}(\yy) = Df(\xx)^{-1}$. By the definition of differentiability, \eqref{eq:linearization} holds (with $Df(\xx)$ in place of $J$) for all sufficiently small $\rr > 0$ and all $\zz \in \BC$ with $\lVert\yy - \zz\rVert \le \delta\rr$. For such $\rr$ and $\zz$, \Cref{lem:linearization} implies that 
\begin{equation*}
    f(A^{-1}(y_{0,\rr}^{j,\bullet}) + \xx) \in B_{\zz,\rr}^{j,\bullet}, \quad \forall j \in [1, a^{-1}]_\BZ, \ \forall \bullet \in \{0, 1, 2\}. 
\end{equation*}
Since $f$ is $\eta$-quasisymmetric, we must have 
\begin{equation*}
    \frac{D_{\Phi_2}(f(A^{-1}(y_{0,\rr}^{j,0}) + \xx), f(A^{-1}(y_{0,\rr}^{j,1}) + \xx))}{D_{\Phi_2}(f(A^{-1}(y_{0,\rr}^{j,0}) + \xx), f(A^{-1}(y_{0,\rr}^{j,2}) + \xx))} \in I_\rr^j, \quad \forall j \in [1, a^{-1}]_\BZ. 
\end{equation*}
Thus, by the definition of the event $F_{\zz,\rr,(I_\rr^1, \dots, I_\rr^{\lfloor a^{-1}\rfloor})}$, we conclude that it does not occur. To lighten notation, we shall write $F_{\zz,\rr} \defeq F_{\zz,\rr,(I_\rr^1, \dots, I_\rr^{\lfloor a^{-1}\rfloor})}$. 

On the other hand, by \Cref{lem:F-properties,lem:independence-across-scales}, and a union bound, we may choose $a$ sufficiently small, depending only on $\eta$, so that with polynomially high probability as $\varepsilon \to 0$, for each $\zz \in B_{1/\varepsilon}(0) \cap (\varepsilon^3\BZ)^2$, there exists $\rr \in [\varepsilon^2, \varepsilon] \cap \{2^{-n}\}_{n \in \BZ}$ such that the event $F_{\zz,\rr}$ occurs. Combining this with the Borel--Cantelli lemma, we obtain that this holds for all sufficiently small $\varepsilon \in \{2^{-n}\}_{n \in \BZ}$. Thus, for each $\yy \in \BC$, there exist arbitrarily small $\rr \in \{2^{-n}\}_{n \in \BZ}$ and $\zz \in \BC$ with $\lVert\yy - \zz\rVert \le \delta\rr$ for which the event $F_{\zz,\rr}$ occurs, contradicting the preceding discussion. This completes the proof. \qed

\subsection{Proofs of \Cref{thm:main-ind,thm:main-aut}}\label{subsection:main-proof}

\begin{lemma}\label{lem:bi-Lipschitz}
    Let $(\widehat\BC, \Phi; 0, 1, \infty)$ be a three-pointed $\sqrt{8/3}$-LQG sphere with unit area. Then, almost surely, for every pair of bounded open subsets $U \Subset V \subset \BC$, the metrics $D_\Phi(\bullet, \bullet; V)|_{U \times U}$ and $D_\Phi(\bullet, \bullet)|_{U \times U}$ are bi-Lipschitz equivalent. The same statement holds with a whole-plane GFF in place of $\Phi$. 
\end{lemma}

\begin{proof}
    Almost surely, for every pair of bounded open subsets $U \Subset V \subset \BC$, there exists $M \ge 1$ such that $D_\Phi(\text{around } V \setminus U) \le M D_\Phi(\partial U, \partial V)$, where $D_\Phi(\text{around } V \setminus U)$ denotes the infimum of the $D_\Phi$-lengths of paths in $V \setminus U$ which disconnect $\partial U$ and $\partial V$. This implies that
    \begin{equation*}
        D_\Phi(\bullet, \bullet)|_{U \times U} \le D_\Phi(\bullet, \bullet; V)|_{U \times U} \le (M + 1) D_\Phi(\bullet, \bullet)|_{U \times U}. 
    \end{equation*}
    This completes the proof. 
\end{proof}

\begin{proof}[Proof of \Cref{thm:main-ind}]
    Let $(\widehat\BC, \Phi_1; 0, 1, \infty)$ and $(\widehat\BC, \Phi_2; 0, 1, \infty)$ be independent three-pointed $\sqrt{8/3}$-LQG spheres with unit area. Let $\widetilde\Phi_1$ and $\widetilde\Phi_2$ be independent whole-plane GFFs. By \Cref{lem:nonexistence}, almost surely, for every open subset $U \subset \BC$, there is no quasisymmetric embedding of $(U, D_{\widetilde\Phi_1})$ into $(\BC, D_{\widetilde\Phi_2})$. Combining this with \Cref{lem:bi-Lipschitz}, almost surely, for every pair of bounded open subsets $U_1 \Subset V_1 \subset \BC$ and $U_2 \Subset V_2 \subset \BC$, there is no quasisymmetric embedding of $(U_1, D_{\widetilde\Phi_1}(\bullet, \bullet; V_1))$ into $(U_2, D_{\widetilde\Phi_2}(\bullet, \bullet; V_2))$. Combining this with \Cref{lem:three-pointed-sphere}, we obtain that, almost surely, for every pair of dyadic domains $U_1 \Subset V_1 \Subset \BC \setminus \{0,1\}$ and $U_2 \Subset V_2 \Subset \BC \setminus \{0,1\}$, there is no quasisymmetric embedding of $(U_1, D_{\Phi_1}(\bullet, \bullet; V_1))$ into $(U_2, D_{\Phi_2}(\bullet, \bullet; V_2))$. (Recall that $U \subset \BC$ is a dyadic domain if there exists a finite collection of dyadic squares such that $U$ is the interior of the union of the closures of these squares.) Applying \Cref{lem:bi-Lipschitz} again, we conclude that, almost surely, for every pair of dyadic domains $U_1 \Subset \BC \setminus \{0,1\}$ and $U_2 \Subset \BC \setminus \{0,1\}$, there is no quasisymmetric embedding of $(U_1, D_{\Phi_1})$ into $(U_2, D_{\Phi_2})$. This implies that, almost surely, there is no quasisymmetric homeomorphism from $(\widehat\BC, D_{\Phi_1})$ to $(\widehat\BC, D_{\Phi_2})$. This completes the proof.
\end{proof}

\begin{proof}[Proof of \Cref{thm:main-aut}]
    Let $(\widehat\BC, \Phi; 0, 1, \infty)$ be a three-pointed $\sqrt{8/3}$-LQG sphere with unit area. Let $\Phi_1$ and $\Phi_2$ be independent whole-plane GFFs. For an open subset $V \subset \BC$, write $\Phi_1 = \Phi_{1,V} + \kH_{1,V}$, where $\Phi_{1,V}$ is a zero-boundary GFF on $V$, $\kH_{1,V}$ is almost surely harmonic on $V$, and $\Phi_{1,V}$ and $\kH_{1,V}$ are independent (cf.~the domain Markov property of the whole-plane GFF). Note that, almost surely, for every bounded open subset $U \Subset V$, since $\kH_{1,V}|_U$ is bounded above and below, the metrics $D_{\Phi_1}(\bullet, \bullet)|_{U \times U}$ and $D_{\Phi_{1,V}}(\bullet, \bullet)|_{U \times U}$ are bi-Lipschitz equivalent. The same statement holds with $\Phi_2$ in place of $\Phi_1$. 

    By \Cref{lem:nonexistence}, almost surely, for every open subset $U \subset \BC$, there is no quasisymmetric embedding of $(U, D_{\Phi_1})$ into $(\BC, D_{\Phi_2})$. Combining this with the discussion in the preceding paragraph, we conclude that, almost surely, for every pair of bounded open subsets $U_1 \Subset V_1 \subset \BC$ and $U_2 \Subset V_2 \subset \BC$, there is no quasisymmetric embedding of $(U_1, D_{\Phi_{1,V_1}})$ into $(U_2, D_{\Phi_{2,V_2}})$. Since $\Phi_{1,V_1}$ and $\Phi_{1,V_2}$ are independent, this implies that, almost surely, for every pair of bounded open subsets $U_1 \Subset V_1 \subset \BC$ and $U_2 \Subset V_2 \subset \BC$ with $V_1$ and $V_2$ disjoint, there is no quasisymmetric embedding of $(U_1, D_{\Phi_{1,V_1}})$ into $(U_2, D_{\Phi_{1,V_2}})$. Since $D_{\Phi_{1,V_1}}(\bullet, \bullet)|_{U_1 \times U_1}$ and $D_{\Phi_1}(\bullet, \bullet)|_{U_1 \times U_1}$ are bi-Lipschitz equivalent, and $D_{\Phi_{1,V_2}}(\bullet, \bullet)|_{U_2 \times U_2}$ and $D_{\Phi_1}(\bullet, \bullet)|_{U_2 \times U_2}$ are bi-Lipschitz equivalent, we deduce that, almost surely, for every pair of bounded open subsets $U_1 \Subset V_1 \subset \BC$ and $U_2 \Subset V_2 \subset \BC$ with $V_1$ and $V_2$ disjoint, there is no quasisymmetric embedding of $(U_1, D_{\Phi_1})$ into $(U_2, D_{\Phi_1})$. Combining this with \Cref{lem:bi-Lipschitz}, almost surely, for every pair of bounded open subsets $U_1 \Subset V_1 \subset \BC$ and $U_2 \Subset V_2 \subset \BC$ with $V_1$ and $V_2$ disjoint, there is no quasisymmetric embedding of $(U_1, D_{\Phi_1}(\bullet, \bullet; V_1))$ into $(U_2, D_{\Phi_1}(\bullet, \bullet; V_2))$. By \Cref{lem:three-pointed-sphere}, it follows that, almost surely, for every pair of dyadic domains $U_1 \Subset V_1 \Subset \BC \setminus \{0, 1\}$ and $U_2 \Subset V_2 \Subset \BC \setminus \{0, 1\}$ with $V_1$ and $V_2$ disjoint, there is no quasisymmetric embedding of $(U_1, D_\Phi(\bullet, \bullet; V_1))$ into $(U_2, D_\Phi(\bullet, \bullet; V_2))$. Applying \Cref{lem:bi-Lipschitz} once more, we obtain that, almost surely, for every pair of dyadic domains $U_1 \Subset V_1 \Subset \BC \setminus \{0, 1\}$ and $U_2 \Subset V_2 \Subset \BC \setminus \{0, 1\}$ with $V_1$ and $V_2$ disjoint, there is no quasisymmetric embedding of $(U_1, D_\Phi)$ into $(U_2, D_\Phi)$. This implies that, almost surely, $(\widehat\BC, D_\Phi)$ has no nontrivial quasisymmetric automorphism. This completes the proof.
\end{proof}

\section{The conformal structure is determined}\label{section:measurability}

In the present section, we sketch how to modify the arguments of \Cref{section:quasiconformality,section:main-proof} to obtain a new proof of \Cref{cor:measurability}.

First, we observe that the argument of \Cref{section:quasiconformality} can be carried out without assuming \Cref{cor:measurability}, provided that one uses the following input instead.

\begin{lemma}
    Let $(\BC, \Phi; 0, \infty)$ be a $\sqrt{8/3}$-LQG cone. Suppose that $\tau$ is a stopping time for the forward or backward filled metric ball exploration. Then the $\sqrt{8/3}$-LQG surfaces parameterized by $B_\tau^\bullet(0; D_\Phi)$ and $\BC \setminus B_\tau^\bullet(0; D_\Phi)$ are conditionally independent given the boundary length of $B_\tau^\bullet(0; D_\Phi)$. (In particular, the conformal structures of $B_\tau^\bullet(0; D_\Phi)$ and $\BC \setminus B_\tau^\bullet(0; D_\Phi)$ are conditionally independent given this boundary length.)
\end{lemma}

This follows from the fact that the filled metric ball exploration of a $\sqrt{8/3}$-LQG cone agrees with the QLE$(8/3, 0)$ process, parameterized by distance, together with the Markovian properties of the forward and backward QLE$(8/3, 0)$ explorations (cf.~\cite{MR4050102,MR4348679}).

The following is an immediate consequence of \Cref{lem:quasiconformality-GFF-GFF}.

\begin{lemma}\label{lem:measurability-quasiconformal}
    Let $\Phi_1$ and $\Phi_2$ be whole-plane GFFs. Suppose that $\Phi_1$ and $\Phi_2$ are coupled so that $(\BC, D_{\Phi_1})$ and $(\BC, D_{\Phi_2})$ are almost surely isometric. Let $f \colon \BC \to \BC$ denote the induced isometry. Then $f$ is almost surely quasiconformal. (Here, we note that there is almost surely a unique isometry between $(\BC, D_{\Phi_1})$ and $(\BC, D_{\Phi_2})$, since their isometry groups are almost surely trivial.)
\end{lemma}

We continue to use the notation of \Cref{lem:measurability-quasiconformal}. By fixing orientations, we may assume without loss of generality that $f$ is orientation-preserving. It remains to show that $f$ is almost surely conformal. Recall that a quasiconformal mapping is differentiable, with invertible Jacobian matrix, at Lebesgue almost every point, and that it is conformal if and only if $\overline\partial f = 0$ Lebesgue almost everywhere. Let $\xx$ be sampled uniformly from Lebesgue measure on $\BD$, independently of everything else. Then $Df(\xx)$ almost surely exists and satisfies $\det(Df(\xx)) > 0$. As in the proof of \Cref{lem:nonexistence}, it suffices to prove the following lemma.

\begin{lemma}\label{lem:measurability-matrix}
    For each invertible matrix $A \in \mathop{\mathrm{GL}}(2, \BR)$ with $\det(A) > 0$ and $A\II \neq \II A$, where $\II \defeq \begin{pmatrix} 0 & -1 \\ 1 & 0 \end{pmatrix}$, there exists $\delta = \delta(A) > 0$ such that, with probability zero, $Df(\xx)$ exists, is invertible, and satisfies $\lVert Df(\xx) - A\rVert \le \delta$.
\end{lemma}

The idea of the proof of \Cref{lem:measurability-matrix} is similar to that of the proof of \Cref{lem:matrix}. Suppose that $Df(\xx)$ exists and is close to $A$. Since $A\II \neq \II A$, the linear map $A$ sends circles to ellipses with nonzero eccentricity. Hence, at all sufficiently small scales, $f$ sends Euclidean circles centered at $\xx$ to subsets close to concentric ellipses with nonzero eccentricity. As in the proof of \Cref{lem:matrix}, we will construct ``patterns'' which rule out this possibility.

The assumptions on $A$ imply that $A(\partial\BD)$ is a centered ellipse with nonzero eccentricity. Let $u$ and $v$ be endpoints of the semi-major and semi-minor axes of this ellipse, respectively. Then $\lVert u\rVert > \lVert v\rVert$. Set $p \defeq A^{-1}(u) \in \partial\BD$ and $q \defeq A^{-1}(v) \in \partial\BD$. Note that $p$ and $q$ are necessarily orthogonal, and so are $u$ and $v$.

\begin{figure}[ht!]
    \centering
    \includegraphics[width=.6\linewidth]{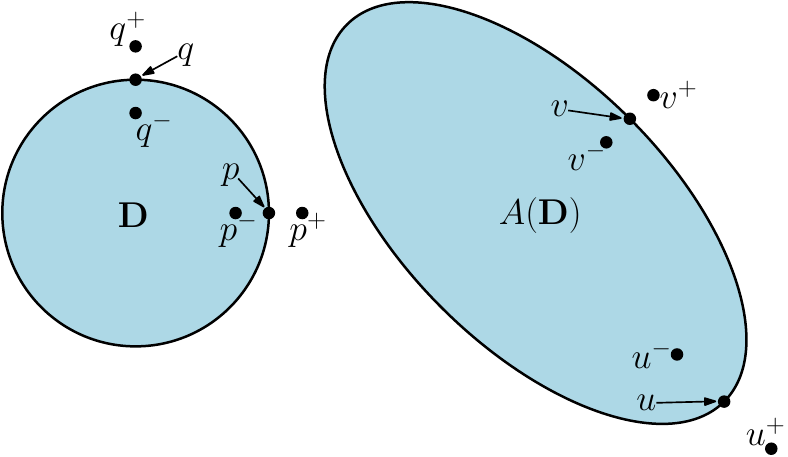}
    \caption{Illustration of the points $p$, $p^-$, $p^+$, $q$, $q^-$, $q^+$, $u$, $u^-$, $u^+$, $v$, $v^-$, and $v^+$ (with $A = \begin{pmatrix} 1 & 1/2 \\ -1 & 1/2 \end{pmatrix}$). The two blue regions depict $\BD$ and $A(\BD)$, respectively.}
    \label{fig:pquv}
\end{figure}

Fix a sufficiently small parameter $b_1 \in (0, 1)$. Set
\begin{gather*}
    p^- \defeq (1 - b_1)p; \quad p^+ \defeq (1 + b_1)p; \quad q^- \defeq (1 - b_1)q; \quad q^+ \defeq (1 + b_1)q; \\
    u^- \defeq (1 - b_1)u; \quad u^+ \defeq (1 + b_1)u; \quad v^- \defeq (1 - b_1)v; \quad v^+ \defeq (1 + b_1)v. 
\end{gather*}
See \Cref{fig:pquv} for an illustration.

\begin{lemma}\label{lem:measurability-pattern}
    We may choose $b_1$ sufficiently small so that there exists a sufficiently small parameter $b_2 \in (0, b_1)$ for which, with probability strictly greater than $1/2$,
    \begin{equation*}
        \inf\{D_{\Phi_2}(x, y) : x \in B_{b_2}(u^-), \ y \in B_{b_2}(u^+)\} > \sup\{D_{\Phi_2}(x, y) : x \in B_{b_2}(v^-), \ y \in B_{b_2}(v^+)\}.
    \end{equation*}
\end{lemma}

\begin{proof}[Proof sketch]
    Fix $b_0 > 0$ sufficiently small so that $B_{b_0\lVert u\rVert}(u)$ and $B_{b_0\lVert v\rVert}(v)$ are disjoint. Write $A_u$ (resp.~$A_v$) for the circle average of $\Phi_2$ over $\partial B_{b_0\lVert u\rVert}(u)$ (resp.~over $\partial B_{b_0\lVert v\rVert}(v)$). Write $\kH_u$ (resp.~$\kH_v$) for the harmonic extension to $B_{b_0\lVert u\rVert}(u)$ (resp.~$B_{b_0\lVert v\rVert}(v)$) of the restriction of $\Phi_2$ to $\BC \setminus B_{b_0\lVert u\rVert}(u)$ (resp.~$\BC \setminus B_{b_0\lVert v\rVert}(v)$). By the domain Markov property of the whole-plane GFF, $\Phi_u \defeq \Phi_2 - \kH_u$ (resp.~$\Phi_v \defeq \Phi_2 - \kH_v$) is a zero-boundary GFF on $B_{b_0\lVert u\rVert}(u)$ (resp.~$B_{b_0\lVert v\rVert}(v)$). Moreover, $\Phi_u$, $\Phi_v$, and $(A_u, A_v)$ are independent.

    We observe that, as $b_1 \to 0$, the pair $(D_{\Phi_2}(u^-, u^+), D_{\Phi_2}(v^-, v^+))$ is asymptotically equivalent in probability to
    \begin{equation*}
        \left(\re^{A_u/\sqrt6} D_{\Phi_u}(u^-, u^+; B_{b_0\lVert u\rVert}(u)), \re^{A_v/\sqrt6} D_{\Phi_v}(v^-, v^+; B_{b_0\lVert v\rVert}(v))\right). 
    \end{equation*}
    Moreover, by the scale and translation invariance of the zero-boundary GFF, together with the scale covariance and translation invariance of the $\sqrt{8/3}$-LQG metric, $D_{\Phi_u}(u^-, u^+; B_{b_0\lVert u\rVert}(u))$ has the same law as $\left(\frac{\lVert u\rVert}{\lVert v\rVert}\right)^{5/6}$ times $D_{\Phi_v}(v^-, v^+; B_{b_0\lVert v\rVert}(v))$. Thus, in order to show that 
    \begin{equation*}
        \BP\lbrack D_{\Phi_2}(u^-, u^+) > D_{\Phi_2}(v^-, v^+)\rbrack > 1/2
    \end{equation*}
    for all sufficiently small $b_1$, one reduces immediately to the following elementary observation: If $X$, $Y$, and $Z$ are independent random variables, $X$ is almost surely positive, $Y$ has the same law as $aX$ for some $a \in (0,1)$, and $Z$ is centered Gaussian, then $\BP\lbrack X > \re^Z Y\rbrack > 1/2$.

    Finally, by continuity of the metric, we may choose $b_2$ sufficiently small relative to $b_1$ so that the desired assertion holds.
\end{proof}

\begin{figure}[ht!]
    \centering
    \includegraphics[width=.9\linewidth]{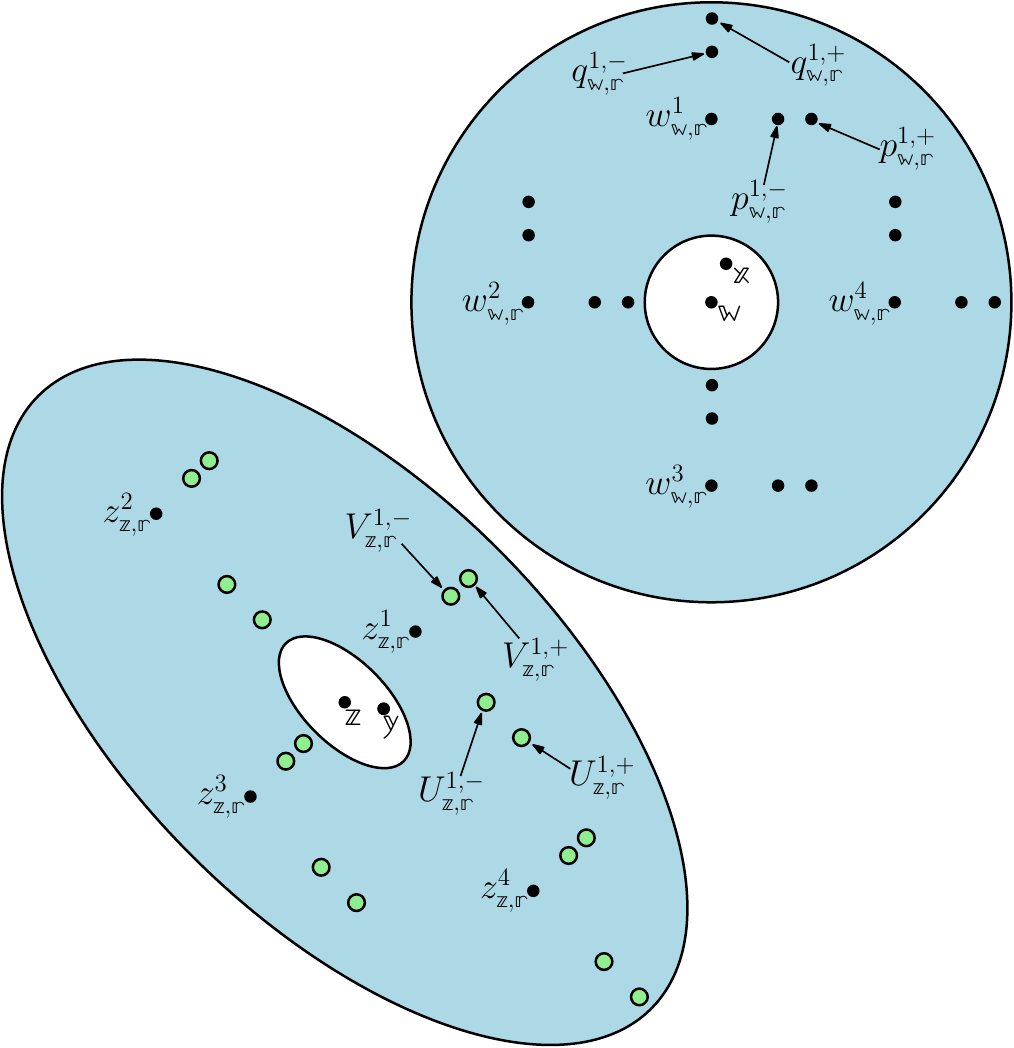}
    \caption{Illustration of various objects involved (with $A = \begin{pmatrix} 1 & 1/2 \\ -1 & 1/2 \end{pmatrix}$ and $a_1 = 1/4$). The two blue regions depict the Euclidean annulus $A_{\rr/2,\rr}(\ww)$ and the elliptical annulus $\zz + A(A_{\rr/2,\rr}(0))$, respectively. Inside $A_{\rr/2,\rr}(\ww)$, there are $a_1^{-1} = 4$ groups of ``test points'' for $\Phi_1$. There are also four groups of ``test points'' inside $\zz + A(A_{\rr/2,\rr}(0))$ for $\Phi_2$. When $\yy = f(\xx)$, $\ww$ and $\zz$ are sufficiently close to $\xx$ and $\yy$, respectively, and $f$ is sufficiently close near $\xx$ to the linear map $\bullet \mapsto A(\bullet - \xx) + \yy$, the points $f(p_{\ww,\rr}^{j,-})$, $f(p_{\ww,\rr}^{j,+})$, $f(q_{\ww,\rr}^{j,-})$, and $f(q_{\ww,\rr}^{j,+})$ must lie in the corresponding green Euclidean balls $U_{\zz,\rr}^{j,-}$, $U_{\zz,\rr}^{j,+}$, $V_{\zz,\rr}^{j,-}$, and $V_{\zz,\rr}^{j,+}$, respectively.}
    \label{fig:conformal-structure}
\end{figure}

Now, we describe how to use the ``pattern'' from \Cref{lem:measurability-pattern} to obtain a contradiction. See \Cref{fig:conformal-structure} for an illustration.

Fix sufficiently small parameters $1 \gg a_1 \gg a_2 > 0$. For $\ww \in \BC$ and $\rr > 0$, as in \Cref{subsection:pattern}, we define the collection of ``test points'' by $w_{\ww,\rr}^j \defeq \ww + (3\rr/4)\re^{2\pi\ri a_1j}$ for $j \in [1, a_1^{-1}]_\BZ$. Set
\begin{equation*}
    p_{\ww,\rr}^{j,-} \defeq w_{\ww,\rr}^j + a_2\rr p^-; \quad p_{\ww,\rr}^{j,+} \defeq w_{\ww,\rr}^j + a_2\rr p^+; \quad q_{\ww,\rr}^{j,-} \defeq w_{\ww,\rr}^j + a_2\rr q^-; \quad q_{\ww,\rr}^{j,+} \defeq w_{\ww,\rr}^j + a_2\rr q^+. 
\end{equation*}
Consider the event that 
\begin{equation}\label{eq:measurability-proof-0}
    D_{\Phi_1}(p_{\ww,\rr}^{j,-}, p_{\ww,\rr}^{j,+}) \le D_{\Phi_1}(q_{\ww,\rr}^{j,-}, q_{\ww,\rr}^{j,+}). 
\end{equation}
The additive constant of $\Phi_1$ does not affect the occurrence of this event. Moreover, by applying a reflection, we may interchange the locations of the pairs $p_{\ww,\rr}^{j,-}, p_{\ww,\rr}^{j,+}$ and $q_{\ww,\rr}^{j,-}, q_{\ww,\rr}^{j,+}$. Hence, by the reflection invariance of the whole-plane GFF, this event has probability at least $1/2$. By choosing $a_2$ sufficiently small relative to $a_1$, we may arrange that these events are sufficiently independent for distinct values of $j$. More precisely, for each $\zeta > 0$, we may choose $a_2$ sufficiently small relative to $a_1$ and design certain events $E_{\ww,\rr}^j$ as in \Cref{subsection:pattern} so that $E_{\ww,\rr}^j$ is contained in the event that $D_{\Phi_1}(p_{\ww,\rr}^{j,-}, p_{\ww,\rr}^{j,+}) \le D_{\Phi_1}(q_{\ww,\rr}^{j,-}, q_{\ww,\rr}^{j,+})$, so that $\BP\lbrack E_{\ww,\rr}^j\rbrack \ge 1/2 - \zeta$, and so that $E_{\ww,\rr}^j$ is almost surely determined by the restriction of $\Phi_1$ to a Euclidean ball centered at $w_{\ww,\rr}^j$, with these balls disjoint for distinct values of $j$. Then, by decreasing $a_1$, we may make the number of trials sufficiently large. By a similar argument to the argument applied in the proof of \Cref{lem:independence-across-balls}, we conclude that for each $\zeta > 0$, we may choose $a_1$ and $a_2$ so that, with probability at least $1 - \zeta$, the number of $j \in [1, a_1^{-1}]_\BZ$ for which $D_{\Phi_1}(p_{\ww,\rr}^{j,-}, p_{\ww,\rr}^{j,+}) \le D_{\Phi_1}(q_{\ww,\rr}^{j,-}, q_{\ww,\rr}^{j,+})$ is at least $(1/2 - \zeta)a_1^{-1}$.

For $\zz \in \BC$, consider the collection of ``test points'' $z_{\zz,\rr}^j \defeq \zz + A(w_{0,\rr}^j)$, $j \in [1, a_1^{-1}]_\BZ$, inside the elliptical annulus $\zz + A(A_{\rr/2,\rr}(0))$. Set
\begin{gather*}
    U_{\zz,\rr}^{j,-} \defeq z_{\zz,\rr}^j + B_{b_2a_2\rr}(a_2\rr u^-); \quad U_{\zz,\rr}^{j,+} \defeq z_{\zz,\rr}^j + B_{b_2a_2\rr}(a_2\rr u^+); \\
    V_{\zz,\rr}^{j,-} \defeq z_{\zz,\rr}^j + B_{b_2a_2\rr}(a_2\rr v^-); \quad V_{\zz,\rr}^{j,+} \defeq z_{\zz,\rr}^j + B_{b_2a_2\rr}(a_2\rr v^+). 
\end{gather*}
Let $\pp > 1/2$ be the probability appearing in \Cref{lem:measurability-pattern}. Then, by the same argument as in the preceding paragraph, for each $\zeta > 0$, by possibly decreasing $a_1$ and $a_2$, we may arrange that, with probability at least $1 - \zeta$, the number of $j \in [1, a_1^{-1}]_\BZ$ for which
\begin{equation}\label{eq:measurability-proof-1}
    \inf\{D_{\Phi_2}(x, y) : x \in U_{\zz,\rr}^{j,-}, \ y \in U_{\zz,\rr}^{j,+}\} > \sup\{D_{\Phi_2}(x, y) : x \in V_{\zz,\rr}^{j,-}, \ y \in V_{\zz,\rr}^{j,+}\}
\end{equation}
is at least $(\pp - \zeta)a_1^{-1}$.

Set $\yy \defeq f(\xx)$. By the same reasoning as in \Cref{lem:linearization}, there exists a sufficiently small $\delta > 0$ such that, whenever $\lVert Df(\xx) - A\rVert \le \delta$, the following holds for all sufficiently small $\rr > 0$ and all $\ww,\zz \in \BC$ with $\lVert\ww - \xx\rVert \le \delta\rr$ and $\lVert\zz - \yy\rVert \le \delta\rr$:
\begin{equation}\label{eq:measurability-proof-2}
    f(p_{\ww,\rr}^{j,-}) \in U_{\zz,\rr}^{j,-}; \quad f(p_{\ww,\rr}^{j,+}) \in U_{\zz,\rr}^{j,+}; \quad f(q_{\ww,\rr}^{j,-}) \in V_{\zz,\rr}^{j,-}; \quad f(q_{\ww,\rr}^{j,+}) \in V_{\zz,\rr}^{j,+}
\end{equation}
for all $j \in [1, a_1^{-1}]_\BZ$.

We observe that if~\eqref{eq:measurability-proof-0},~\eqref{eq:measurability-proof-1}, and~\eqref{eq:measurability-proof-2} all hold for the same values of $\rr$ and $j$, then we obtain a contradiction to the fact that $f$ is an isometry. Indeed, by choosing $\zeta$ sufficiently small and applying \Cref{lem:independence-across-scales} together with the Borel--Cantelli lemma, we may arrange that the following holds for all sufficiently small $\varepsilon \in \{2^{-n}\}_{n \in \BZ}$. For each $\zz,\ww \in B_{1/\varepsilon}(0) \cap (\varepsilon^3\BZ)^2$, there are at least $(2/3)\log_2(1/\varepsilon)$ values of $\rr \in [\varepsilon^2, \varepsilon] \cap \{2^{-n}\}_{n \in \BZ}$ for which the number of $j \in [1, a_1^{-1}]_\BZ$ satisfying~\eqref{eq:measurability-proof-0} is at least $(1/2 - \zeta)a_1^{-1}$. There are also at least $(2/3)\log_2(1/\varepsilon)$ values of $\rr \in [\varepsilon^2, \varepsilon] \cap \{2^{-n}\}_{n \in \BZ}$ for which the number of $j \in [1, a_1^{-1}]_\BZ$ satisfying~\eqref{eq:measurability-proof-1} is at least $(\pp - \zeta)a_1^{-1}$. Since $2/3 + 2/3 > 1$, there exists a value of $\rr$ for which both conclusions hold. Since $\pp > 1/2$, we may choose $\zeta$ so that $(1/2 - \zeta) + (\pp - \zeta) > 1$. Therefore, for this value of $\rr$, there exists $j \in [1, a_1^{-1}]_\BZ$ for which both~\eqref{eq:measurability-proof-0} and~\eqref{eq:measurability-proof-1} hold. Together with~\eqref{eq:measurability-proof-2}, this contradicts the fact that $f$ is an isometry.

Thus, we have proved that the conformal structure of the $\sqrt{8/3}$-LQG surface represented by $(\BC,\Phi)$, where $\Phi$ is a whole-plane GFF, is almost surely determined by its metric structure, up to complex conjugation. As in \Cref{subsection:main-proof}, we may transfer this result to the $\sqrt{8/3}$-LQG sphere using local absolute continuity.

\appendix

\section{Measurability of quasisymmetric embeddings}

In the present appendix, we prove the following measurability result concerning the existence of quasisymmetric embeddings.

\begin{lemma}\label{lem:Borel}
    Let $(\MM, \dd_{\mathrm{GH}})$ be the Gromov--Hausdorff space of nonempty compact metric spaces. Fix a distortion function $\eta$. Then the following subset of $\MM \times \MM$ is Borel:
    \begin{multline}\label{eq:Borel}
        \{((X, D_X), (Y, D_Y)) \in \MM \times \MM : X \text{ is not a singleton, and} \\
        \text{there exists an } \eta\text{-QS embedding of } (X, D_X) \text{ into } (Y, D_Y)\}. 
    \end{multline}
\end{lemma}

\begin{proof}
    Let $a > 0$. Write $E_{\eta,a} \subset \MM \times \MM$ for the subset of pairs $((X, D_X), (Y, D_Y))$ for which $X$ is not a singleton, and there exists an $\eta$-quasisymmetric embedding $f \colon (X, D_X) \to (Y, D_Y)$ and points $p,q \in X$ such that $D_X(p, q) \ge a$ and $D_Y(f(p), f(q)) \ge a$. Since the subset in~\eqref{eq:Borel} is equal to $\bigcup_{a \in \BQ_{>0}} E_{\eta,a}$, it suffices to show that $E_{\eta,a}$ is closed in $\MM \times \MM$ for each $a > 0$.

    Suppose that $(X_n, D_{X_n}) \to (X, D_X)$ and $(Y_n, D_{Y_n}) \to (Y, D_Y)$ as $n \to \infty$ with respect to the Gromov--Hausdorff distance. Suppose also that $((X_n, D_{X_n}), (Y_n, D_{Y_n})) \in E_{\eta,a}$ for every $n$. For each $n$, let $f_n \colon (X_n, D_{X_n}) \to (Y_n, D_{Y_n})$ be an $\eta$-quasisymmetric embedding, and let $p_n, q_n \in X_n$ be such that $D_{X_n}(p_n, q_n) \ge a$ and $D_{Y_n}(f_n(p_n), f_n(q_n)) \ge a$. It suffices to show that $((X, D_X), (Y, D_Y)) \in E_{\eta,a}$. 
    
    We first record several simple observations. 
    \begin{itemize}
        \item Write $D$ for the supremum of $\diam(X_n, D_{X_n})$ and $\diam(Y_n, D_{Y_n})$ over all $n$. Since the sequences $(X_n, D_{X_n})$ and $(Y_n, D_{Y_n})$ converge with respect to the Gromov--Hausdorff distance, we have $D < \infty$.
        \item Let $x_n^1, x_n^2 \in X_n$. Then there exists $r_n \in \{p_n,q_n\}$ such that $D_{X_n}(x_n^1,r_n) \ge a/2$. Since $f_n$ is $\eta$-quasisymmetric, we have
        \begin{equation}\label{eq:Borel-proof-0}
            D_{Y_n}(f_n(x_n^1), f_n(x_n^2)) \le \eta\!\left(\frac{D_{X_n}(x_n^1, x_n^2)}{D_{X_n}(x_n^1, r_n)}\right)D_{Y_n}(f_n(x_n^1), f_n(r_n)) \le \eta(2a^{-1}D_{X_n}(x_n^1, x_n^2)) D.
        \end{equation}
        \item Let $x_n^1, x_n^2 \in X_n$ be distinct. Then
        \begin{equation*}
            a \le D_{Y_n}(f_n(p_n), f_n(q_n)) \le D_{Y_n}(f_n(p_n), f_n(x_n^1)) + D_{Y_n}(f_n(x_n^1), f_n(x_n^2)) + D_{Y_n}(f_n(x_n^2), f_n(q_n)). 
        \end{equation*}
        Moreover, by the $\eta$-quasisymmetry of $f_n$,
        \begin{align*}
            D_{Y_n}(f_n(p_n), f_n(x_n^1)) &\le \eta\!\left(\frac{D_{X_n}(p_n, x_n^1)}{D_{X_n}(x_n^1, x_n^2)}\right) D_{Y_n}(f_n(x_n^1), f_n(x_n^2)) \\
            &\le \eta(D/D_{X_n}(x_n^1, x_n^2)) D_{Y_n}(f_n(x_n^1), f_n(x_n^2)).
        \end{align*}
        Similarly, $D_{Y_n}(f_n(x_n^2), f_n(q_n)) \le \eta(D/D_{X_n}(x_n^1, x_n^2)) D_{Y_n}(f_n(x_n^1), f_n(x_n^2))$. Therefore,
        \begin{equation}\label{eq:Borel-proof-1}
            D_{Y_n}(f_n(x_n^1), f_n(x_n^2)) \ge \frac a{1 + 2\eta(D/D_{X_n}(x_n^1, x_n^2))}. 
        \end{equation}
    \end{itemize}

    By the standard realization theorem for Gromov--Hausdorff convergence, there exist compact metric spaces $(\widetilde X, D_{\widetilde X})$ and $(\widetilde Y, D_{\widetilde Y})$, together with isometric embeddings of $(X_n, D_{X_n})$ and $(X, D_X)$ into $(\widetilde X, D_{\widetilde X})$, and of $(Y_n, D_{Y_n})$ and $(Y, D_Y)$ into $(\widetilde Y, D_{\widetilde Y})$, such that $X_n \to X$ and $Y_n \to Y$ as $n \to \infty$ in the Hausdorff distance. By passing to a subsequence, we may assume that the graphs $G_n \defeq \{(x, f_n(x)) \in \widetilde X \times \widetilde Y : x \in X_n\}$ converge to some compact subset $G \subset \widetilde X \times \widetilde Y$ in the Hausdorff distance, and that $p_n$ and $q_n$ converge to some points $p,q \in \widetilde X$, respectively. It is clear that $G \subset X \times Y$ and $p,q \in X$.

    We first verify that $G$ is the graph of a mapping $f \colon X \to Y$, i.e., for each $x \in X$, there exists a unique $y \in Y$ with $(x, y) \in G$.  Let $x \in X$. Since $X_n \to X$ in the Hausdorff distance, there exists a sequence $x_n \in X_n$ such that $x_n \to x$. By passing to a subsequence, we may assume that $(x_n, f_n(x_n))$ converges to some $(x,y) \in G$. This proves existence. For uniqueness, suppose by contradiction that there exist distinct points $(x,y^1),(x,y^2) \in G$. Then there exist sequences $(x_n^1, f_n(x_n^1)) \to (x,y^1)$ and $(x_n^2, f_n(x_n^2)) \to (x,y^2)$ with $x_n^1,x_n^2 \in X_n$. However, by~\eqref{eq:Borel-proof-0} and the fact that $D_{X_n}(x_n^1,x_n^2) \to 0$, we have $D_{Y_n}(f_n(x_n^1), f_n(x_n^2)) \to 0$, which contradicts $y^1 \neq y^2$.

    Next, we verify that $f$ is injective. Let $x^1, x^2 \in X$ be distinct. Then there exist sequences $(x_n^1, f_n(x_n^1)) \to (x^1, f(x^1))$ and $(x_n^2, f_n(x_n^2)) \to (x^2, f(x^2))$ with $x_n^1, x_n^2 \in X_n$. Choose $\varepsilon > 0$ such that $D_{X_n}(x_n^1, x_n^2) \ge \varepsilon$ for all sufficiently large $n$. By~\eqref{eq:Borel-proof-1}, we have $D_{Y_n}(f_n(x_n^1), f_n(x_n^2)) \ge a/(1 + 2\eta(D/\varepsilon))$ for all sufficiently large $n$. Passing to the limit gives $f(x^1) \neq f(x^2)$.

    We now verify that $f$ is $\eta$-quasisymmetric. Let $x^1,x^2,x^3 \in X$ be distinct. Then there exist sequences $(x_n^1, f_n(x_n^1)) \to (x^1, f(x^1))$, $(x_n^2, f_n(x_n^2)) \to (x^2, f(x^2))$, and $(x_n^3, f_n(x_n^3)) \to (x^3, f(x^3))$ with $x_n^1,x_n^2,x_n^3 \in X_n$. Since $f_n$ is $\eta$-quasisymmetric, we have
    \begin{equation*}
        \frac{D_{Y_n}(f_n(x_n^1), f_n(x_n^2))}{D_{Y_n}(f_n(x_n^1), f_n(x_n^3))} \le \eta\!\left(\frac{D_{X_n}(x_n^1, x_n^2)}{D_{X_n}(x_n^1, x_n^3)}\right). 
    \end{equation*}
    Passing to the limit gives
    \begin{equation*}
        \frac{D_Y(f(x^1), f(x^2))}{D_Y(f(x^1), f(x^3))} \le \eta\!\left(\frac{D_X(x^1, x^2)}{D_X(x^1, x^3)}\right).
    \end{equation*}

    Finally, since $p_n \to p$ and $q_n \to q$, we have $D_X(p, q) \ge a$. Moreover, since $G$ is the graph of $f$, every subsequential limit of $f_n(p_n)$ must be equal to $f(p)$, and similarly every subsequential limit of $f_n(q_n)$ must be equal to $f(q)$. Therefore, passing to the limit in $D_{Y_n}(f_n(p_n), f_n(q_n)) \ge a$ gives $D_Y(f(p), f(q)) \ge a$. This completes the proof that $((X, D_X), (Y, D_Y)) \in E_{\eta,a}$, hence also the proof of \Cref{lem:Borel}. 
\end{proof}

\bibliographystyle{alpha}
\bibliography{references}

\newcommand{\etalchar}[1]{$^{#1}$}
\begin{thebibliography}{DFG{\etalchar{+}}20}

\bibitem[AB60]{MR115006}
Lars Ahlfors and Lipman Bers.
\newblock Riemann's mapping theorem for variable metrics.
\newblock {\em Ann. of Math. (2)}, 72:385--404, 1960.

\bibitem[ABA17]{MR3706731}
Louigi Addario-Berry and Marie Albenque.
\newblock The scaling limit of random simple triangulations and random simple quadrangulations.
\newblock {\em Ann. Probab.}, 45(5):2767--2825, 2017.

\bibitem[ABA21]{MR4315765}
Louigi Addario-Berry and Marie Albenque.
\newblock Convergence of non-bipartite maps via symmetrization of labeled trees.
\newblock {\em Ann. H. Lebesgue}, 4:653--683, 2021.

\bibitem[Abr16]{MR3498001}
C\'eline Abraham.
\newblock Rescaled bipartite planar maps converge to the {B}rownian map.
\newblock {\em Ann. Inst. Henri Poincar\'e{} Probab. Stat.}, 52(2):575--595, 2016.

\bibitem[Ahl73]{CoInAhl}
Lars~V. Ahlfors.
\newblock {\em Conformal invariants: topics in geometric function theory}.
\newblock McGraw-Hill Series in Higher Mathematics. McGraw-Hill Book Co., New York-D\"usseldorf-Johannesburg, 1973.

\bibitem[Ahl06]{MR2241787}
Lars~V. Ahlfors.
\newblock {\em Lectures on quasiconformal mappings}, volume~38 of {\em University Lecture Series}.
\newblock American Mathematical Society, Providence, RI, second edition, 2006.
\newblock With supplemental chapters by C. J. Earle, I. Kra, M. Shishikura and J. H. Hubbard.

\bibitem[BBM21]{MR4273635}
James Belk, Collin Bleak, and Francesco Matucci.
\newblock Rational embeddings of hyperbolic groups.
\newblock {\em J. Comb. Algebra}, 5(2):123--183, 2021.

\bibitem[BC25]{2025arXiv250818792B}
Timothy {Budd} and Nicolas {Curien}.
\newblock {Random punctured hyperbolic surfaces \& the Brownian sphere}.
\newblock {\em arXiv e-prints}, page arXiv:2508.18792, August 2025.

\bibitem[BF14]{MR3445628}
Bodil Branner and N\'uria Fagella.
\newblock {\em Quasiconformal surgery in holomorphic dynamics}, volume 141 of {\em Cambridge Studies in Advanced Mathematics}.
\newblock Cambridge University Press, Cambridge, 2014.
\newblock With contributions by Xavier Buff, Shaun Bullett, Adam L. Epstein, Peter Ha\"issinsky, Christian Henriksen, Carsten L. Petersen, Kevin M. Pilgrim, Tan Lei and Michael Yampolsky.

\bibitem[BGS25]{BorgaGwynneSun2025}
Jacopo Borga, Ewain Gwynne, and Xin Sun.
\newblock Permutons, meanders, and {SLE}-decorated {Liouville} quantum gravity.
\newblock {\em J. Eur. Math. Soc.}, 2025.
\newblock Published online first.

\bibitem[BJM14]{MR3256874}
J\'er\'emie Bettinelli, Emmanuel Jacob, and Gr\'egory Miermont.
\newblock The scaling limit of uniform random plane maps, {\it via} the {A}mbj\o rn-{B}udd bijection.
\newblock {\em Electron. J. Probab.}, 19:no. 74, 16, 2014.

\bibitem[BK02]{MR1930885}
Mario Bonk and Bruce Kleiner.
\newblock Quasisymmetric parametrizations of two-dimensional metric spheres.
\newblock {\em Invent. Math.}, 150(1):127--183, 2002.

\bibitem[BK05]{MR2116315}
Mario Bonk and Bruce Kleiner.
\newblock Conformal dimension and {G}romov hyperbolic groups with 2-sphere boundary.
\newblock {\em Geom. Topol.}, 9:219--246, 2005.

\bibitem[BKM09]{MR2503988}
Mario Bonk, Bruce Kleiner, and Sergei Merenkov.
\newblock Rigidity of {S}chottky sets.
\newblock {\em Amer. J. Math.}, 131(2):409--443, 2009.

\bibitem[BM13]{MR3010807}
Mario Bonk and Sergei Merenkov.
\newblock Quasisymmetric rigidity of square {S}ierpi\'nski carpets.
\newblock {\em Ann. of Math. (2)}, 177(2):591--643, 2013.

\bibitem[BM20]{MR4140760}
Mario Bonk and Sergei Merenkov.
\newblock Square {S}ierpi\'{n}ski carpets and {L}att\`es maps.
\newblock {\em Math. Z.}, 296(1-2):695--718, 2020.

\bibitem[Bon11]{MR2854086}
Mario Bonk.
\newblock Uniformization of {S}ierpi\'nski carpets in the plane.
\newblock {\em Invent. Math.}, 186(3):559--665, 2011.

\bibitem[BP02]{MR1919393}
Marc Bourdon and Herv\'e Pajot.
\newblock Quasi-conformal geometry and hyperbolic geometry.
\newblock In {\em Rigidity in dynamics and geometry ({C}ambridge, 2000)}, pages 1--17. Springer, Berlin, 2002.

\bibitem[BS00]{MR1771428}
M.~Bonk and O.~Schramm.
\newblock Embeddings of {G}romov hyperbolic spaces.
\newblock {\em Geom. Funct. Anal.}, 10(2):266--306, 2000.

\bibitem[Can94]{MR1301392}
James~W. Cannon.
\newblock The combinatorial {R}iemann mapping theorem.
\newblock {\em Acta Math.}, 173(2):155--234, 1994.

\bibitem[CG93]{MR1230383}
Lennart Carleson and Theodore~W. Gamelin.
\newblock {\em Complex dynamics}.
\newblock Universitext: Tracts in Mathematics. Springer-Verlag, New York, 1993.

\bibitem[CLG14]{TBP}
Nicolas Curien and Jean-Fran{\c{c}}ois Le~Gall.
\newblock The {B}rownian plane.
\newblock {\em J. Theoret. Probab.}, 27(4):1249--1291, 2014.

\bibitem[CLG16]{Hull-TBP}
Nicolas Curien and Jean-Fran{\c{c}}ois Le~Gall.
\newblock The hull process of the {B}rownian plane.
\newblock {\em Probab. Theory Related Fields}, 166(1-2):187--231, 2016.

\bibitem[DDDF20]{TightLFPP}
Jian Ding, Julien Dub\'edat, Alexander Dunlap, and Hugo Falconet.
\newblock Tightness of {L}iouville first passage percolation for {$\gamma \in (0,2)$}.
\newblock {\em Publ. Math. Inst. Hautes \'Etudes Sci.}, 132:353--403, 2020.

\bibitem[DFG{\etalchar{+}}20]{WeakLQGMet}
Julien Dub\'edat, Hugo Falconet, Ewain Gwynne, Joshua Pfeffer, and Xin Sun.
\newblock Weak {LQG} metrics and {L}iouville first passage percolation.
\newblock {\em Probab. Theory Related Fields}, 178(1-2):369--436, 2020.

\bibitem[DMS21]{dms2021mating}
Bertrand Duplantier, Jason Miller, and Scott Sheffield.
\newblock Liouville quantum gravity as a mating of trees.
\newblock {\em Ast\'{e}risque}, (427):viii+257, 2021.

\bibitem[DS11]{LQGKPZ}
Bertrand Duplantier and Scott Sheffield.
\newblock Liouville quantum gravity and {KPZ}.
\newblock {\em Invent. Math.}, 185(2):333--393, 2011.

\bibitem[GL00]{MR1730906}
Frederick~P. Gardiner and Nikola Lakic.
\newblock {\em Quasiconformal {T}eichm\"uller theory}, volume~76 of {\em Mathematical Surveys and Monographs}.
\newblock American Mathematical Society, Providence, RI, 2000.

\bibitem[GM20a]{ConfLQG}
Ewain Gwynne and Jason Miller.
\newblock Confluence of geodesics in {L}iouville quantum gravity for {$\gamma \in (0, 2)$}.
\newblock {\em Ann. Probab.}, 48(4):1861--1901, 2020.

\bibitem[GM20b]{LocMetGFF}
Ewain Gwynne and Jason Miller.
\newblock Local metrics of the {G}aussian free field.
\newblock {\em Ann. Inst. Fourier (Grenoble)}, 70(5):2049--2075, 2020.

\bibitem[GM21]{ExUniLQG}
Ewain Gwynne and Jason Miller.
\newblock Existence and uniqueness of the {L}iouville quantum gravity metric for {$\gamma\in(0,2)$}.
\newblock {\em Invent. Math.}, 223(1):213--333, 2021.

\bibitem[GMS20]{TuttePVBDtoLQG}
Ewain Gwynne, Jason Miller, and Scott Sheffield.
\newblock The {T}utte embedding of the {P}oisson-{V}oronoi tessellation of the {B}rownian disk converges to {$\sqrt{8/3}$}-{L}iouville quantum gravity.
\newblock {\em Comm. Math. Phys.}, 374(2):735--784, 2020.

\bibitem[GP91]{MR1168043}
M.~Gromov and P.~Pansu.
\newblock Rigidity of lattices: an introduction.
\newblock In {\em Geometric topology: recent developments ({M}ontecatini {T}erme, 1990)}, volume 1504 of {\em Lecture Notes in Math.}, pages 39--137. Springer, Berlin, 1991.

\bibitem[Gro87]{MR919829}
M.~Gromov.
\newblock Hyperbolic groups.
\newblock In {\em Essays in group theory}, volume~8 of {\em Math. Sci. Res. Inst. Publ.}, pages 75--263. Springer, New York, 1987.

\bibitem[Hei01]{MR1800917}
Juha Heinonen.
\newblock {\em Lectures on analysis on metric spaces}.
\newblock Universitext. Springer-Verlag, New York, 2001.

\bibitem[HK95]{MR1323982}
Juha Heinonen and Pekka Koskela.
\newblock Definitions of quasiconformality.
\newblock {\em Invent. Math.}, 120(1):61--79, 1995.

\bibitem[HK98]{MR1654771}
Juha Heinonen and Pekka Koskela.
\newblock Quasiconformal maps in metric spaces with controlled geometry.
\newblock {\em Acta Math.}, 181(1):1--61, 1998.

\bibitem[IT92]{MR1215481}
Y.~Imayoshi and M.~Taniguchi.
\newblock {\em An introduction to {T}eichm\"uller spaces}.
\newblock Springer-Verlag, Tokyo, 1992.
\newblock Translated and revised from the Japanese by the authors.

\bibitem[KK00a]{MR1771571}
Sari Kallunki and Pekka Koskela.
\newblock Exceptional sets for the definition of quasiconformality.
\newblock {\em Amer. J. Math.}, 122(4):735--743, 2000.

\bibitem[KK00b]{MR1834498}
Michael Kapovich and Bruce Kleiner.
\newblock Hyperbolic groups with low-dimensional boundary.
\newblock {\em Ann. Sci. \'Ecole Norm. Sup. (4)}, 33(5):647--669, 2000.

\bibitem[LG07]{MR2336042}
Jean-Fran{\c{c}}ois Le~Gall.
\newblock The topological structure of scaling limits of large planar maps.
\newblock {\em Invent. Math.}, 169(3):621--670, 2007.

\bibitem[LG10]{GeoBMap}
Jean-Fran{\c{c}}ois Le~Gall.
\newblock Geodesics in large planar maps and in the {B}rownian map.
\newblock {\em Acta Math.}, 205(2):287--360, 2010.

\bibitem[LG13]{MR3112934}
Jean-Fran{\c{c}}ois Le~Gall.
\newblock Uniqueness and universality of the {B}rownian map.
\newblock {\em Ann. Probab.}, 41(4):2880--2960, 2013.

\bibitem[LG22]{MR4474537}
Jean-Fran{\c{c}}ois Le~Gall.
\newblock The volume measure of the {B}rownian sphere is a {H}ausdorff measure.
\newblock {\em Electron. J. Probab.}, 27:Paper No. 113, 28, 2022.

\bibitem[LGP08]{MR2438999}
Jean-Fran{\c{c}}ois Le~Gall and Fr\'ed\'eric Paulin.
\newblock Scaling limits of bipartite planar maps are homeomorphic to the 2-sphere.
\newblock {\em Geom. Funct. Anal.}, 18(3):893--918, 2008.

\bibitem[LGR21]{MR4341082}
Jean-Fran{\c c}ois Le~Gall and Armand Riera.
\newblock Spine representations for non-compact models of random geometry.
\newblock {\em Probab. Theory Related Fields}, 181(1-3):571--645, 2021.

\bibitem[LV73]{MR344463}
O.~Lehto and K.~I. Virtanen.
\newblock {\em Quasiconformal mappings in the plane}, volume Band 126 of {\em Die Grundlehren der mathematischen Wissenschaften}.
\newblock Springer-Verlag, New York-Heidelberg, second edition, 1973.
\newblock Translated from the German by K. W. Lucas.

\bibitem[LW20]{MR4073230}
Alexander Lytchak and Stefan Wenger.
\newblock Canonical parameterizations of metric disks.
\newblock {\em Duke Math. J.}, 169(4):761--797, 2020.

\bibitem[McM99]{MR1726737}
Curtis~T. McMullen.
\newblock Hausdorff dimension and conformal dynamics. {I}. {S}trong convergence of {K}leinian groups.
\newblock {\em J. Differential Geom.}, 51(3):471--515, 1999.

\bibitem[Mie08]{MR2399286}
Gr\'egory Miermont.
\newblock On the sphericity of scaling limits of random planar quadrangulations.
\newblock {\em Electron. Commun. Probab.}, 13:248--257, 2008.

\bibitem[Mie13]{MR3070569}
Gr\'egory Miermont.
\newblock The {B}rownian map is the scaling limit of uniform random plane quadrangulations.
\newblock {\em Acta Math.}, 210(2):319--401, 2013.

\bibitem[Mos68]{MR236383}
G.~D. Mostow.
\newblock Quasi-conformal mappings in {$n$}-space and the rigidity of hyperbolic space forms.
\newblock {\em Inst. Hautes \'Etudes Sci. Publ. Math.}, (34):53--104, 1968.

\bibitem[MS16]{MR3572845}
Jason Miller and Scott Sheffield.
\newblock Quantum {L}oewner evolution.
\newblock {\em Duke Math. J.}, 165(17):3241--3378, 2016.

\bibitem[MS17]{IG4}
Jason Miller and Scott Sheffield.
\newblock Imaginary geometry {IV}: interior rays, whole-plane reversibility, and space-filling trees.
\newblock {\em Probab. Theory Related Fields}, 169(3-4):729--869, 2017.

\bibitem[MS20]{MR4050102}
Jason Miller and Scott Sheffield.
\newblock Liouville quantum gravity and the {B}rownian map {I}: the {${\rm QLE}(8/3,0)$} metric.
\newblock {\em Invent. Math.}, 219(1):75--152, 2020.

\bibitem[MS21a]{MR4225028}
Jason Miller and Scott Sheffield.
\newblock An axiomatic characterization of the {B}rownian map.
\newblock {\em J. \'Ec. polytech. Math.}, 8:609--731, 2021.

\bibitem[MS21b]{MR4348679}
Jason Miller and Scott Sheffield.
\newblock Liouville quantum gravity and the {B}rownian map {II}: {G}eodesics and continuity of the embedding.
\newblock {\em Ann. Probab.}, 49(6):2732--2829, 2021.

\bibitem[MS21c]{MR4242633}
Jason Miller and Scott Sheffield.
\newblock Liouville quantum gravity and the {B}rownian map {III}: the conformal structure is determined.
\newblock {\em Probab. Theory Related Fields}, 179(3-4):1183--1211, 2021.

\bibitem[MT26]{2026arXiv260324473M}
Jason {Miller} and Yi~{Tian}.
\newblock {The conformal dimension of the Brownian sphere is two}.
\newblock {\em arXiv e-prints}, page arXiv:2603.24473, March 2026.

\bibitem[MW25]{MR4956983}
Damaris Meier and Stefan Wenger.
\newblock Quasiconformal almost parametrizations of metric surfaces.
\newblock {\em J. Eur. Math. Soc. (JEMS)}, 27(12):5133--5154, 2025.

\bibitem[NR23]{MR4608329}
Dimitrios Ntalampekos and Matthew Romney.
\newblock Polyhedral approximation of metric surfaces and applications to uniformization.
\newblock {\em Duke Math. J.}, 172(9):1673--1734, 2023.

\bibitem[Pan89]{MR1024425}
Pierre Pansu.
\newblock Dimension conforme et sph\`ere \`a{} l'infini des vari\'et\'es \`a{} courbure n\'egative.
\newblock {\em Ann. Acad. Sci. Fenn. Ser. A I Math.}, 14(2):177--212, 1989.

\bibitem[Pol81]{Polyakov:1981rd}
Alexander~M. Polyakov.
\newblock {Quantum Geometry of Bosonic Strings}.
\newblock {\em Phys. Lett. B}, 103:207--210, 1981.

\bibitem[Raj17]{MR3608292}
Kai Rajala.
\newblock Uniformization of two-dimensional metric surfaces.
\newblock {\em Invent. Math.}, 207(3):1301--1375, 2017.

\bibitem[Sul85a]{MR819553}
Dennis Sullivan.
\newblock Quasiconformal homeomorphisms and dynamics. {I}. {S}olution of the {F}atou-{J}ulia problem on wandering domains.
\newblock {\em Ann. of Math. (2)}, 122(3):401--418, 1985.

\bibitem[Sul85b]{MR806415}
Dennis Sullivan.
\newblock Quasiconformal homeomorphisms and dynamics. {II}. {S}tructural stability implies hyperbolicity for {K}leinian groups.
\newblock {\em Acta Math.}, 155(3-4):243--260, 1985.

\bibitem[Tro21]{MR4242630}
Sascha Troscheit.
\newblock On quasisymmetric embeddings of the {B}rownian map and continuum trees.
\newblock {\em Probab. Theory Related Fields}, 179(3-4):1023--1046, 2021.

\bibitem[TV80]{MR595180}
P.~Tukia and J.~V\"ais\"al\"a.
\newblock Quasisymmetric embeddings of metric spaces.
\newblock {\em Ann. Acad. Sci. Fenn. Ser. A I Math.}, 5(1):97--114, 1980.

\bibitem[V{\"a}i71]{MR454009}
Jussi V{\"a}is{\"a}l{\"a}.
\newblock {\em Lectures on {$n$}-dimensional quasiconformal mappings}, volume Vol. 229 of {\em Lecture Notes in Mathematics}.
\newblock Springer-Verlag, Berlin-New York, 1971.

\end{thebibliography}

\end{document}